\documentclass[12pt]{amsart}
\usepackage{latexsym,amssymb,amsmath}
\usepackage{amsmath,amsfonts}
\usepackage{epsfig}
\usepackage{graphicx}
\usepackage{verbatim}
\newtheorem{theorem}{Theorem}[section]
\newtheorem{lemma}[theorem]{Lemma}
\newtheorem{proposition}[theorem]{Proposition}

\newtheorem{definition}[theorem]{Definition}
\newtheorem{rmk}[theorem]{Remark}

\allowdisplaybreaks

\def\T{\mathbb T}

\def\R{{\mathbb R}}
\def\Z{{\mathbb Z}}
\def\la{\langle}
\def\ra{\rangle}

\def\les{\lesssim}
\def\1{{\bf 1}}
\def\eqnn{\begin{eqnarray*}}
\def\eeqnn{\end{eqnarray*}}
\def\eqn{\begin{eqnarray}}
\def\eeqn{\end{eqnarray}}
\newcommand{\nc}{\newcommand}
\nc{\be}{\begin{equation}}
\nc{\ee}{\end{equation}}
\nc{\ba}{\begin{eqnarray}}
\nc{\ea}{\end{eqnarray}}
\nc{\eps}{\epsilon}
\def\prf{\begin{proof}}
\def\endprf{\end{proof}}

\begin{document}

\title[Regularity properties of Zakharov system on $\R^+$]{Regularity properties of the Zakharov system on the half line} 
 
\author{{\bf M.~B.~Erdo\u gan, N.~Tzirakis}\\
University of Illinois\\
Urbana-Champaign}

\thanks{Email addresses: berdogan@math.uiuc.edu (B. Erdogan), tzirakis@math.uiuc.edu (N. Tzirakis)}\thanks{The first author was partially supported by NSF grant DMS-1501041. The second author's work was supported by a grant from the Simons Foundation (\#355523 Nikolaos Tzirakis).}

\date{}

\begin{abstract}
In this paper we study the local and global regularity properties of the Zakharov system on the half line with rough initial data. These properties include local and global wellposedness results, local and global smoothing results and the behavior of higher order Sobolev norms of the solutions.    Smoothing means that the nonlinear part of the solution on the half line is smoother than the initial data. The gain in regularity coincides with the gain that was observed for the periodic Zakharov \cite{et3} and the Zakharov on the real line. Uniqueness is proved in the class of smooth solutions. When the boundary value of the Schr\"odinger part  of the solution is zero, uniqueness can be extended to the full range of local solutions. Under the same assumptions on the initial data  we also prove global-in-time existence and uniqueness of energy solutions. For more regular data we prove that all higher Sobolev norms grow at most polynomially-in-time. 
\end{abstract}

\maketitle
\section{Introduction}
The Zakharov system is a system of non-linear partial differential equations, introduced by Zakharov in 1972, \cite{vz}.  It describes the propagation of Langmuir waves in an ionized plasma. The system   consists of a complex field $u$ (Schr\"odinger part) and a real field $n$ (wave part) satisfying the equation:
\begin{equation}\label{eq:zakharov}
\left\{
\begin{array}{l}
iu_{t}+  u_{xx} =nu, \,\,\,\,  x \in { \R}^+=(0,\infty), \,\,\,\,  t\in \R^+,\\
n_{tt}-n_{xx}=(|u|^2)_{xx}, \\
u(x,0)=g(x)\in H^{s_0}({ \R}^+), \\
n(x,0)=n_0(x)\in H^{s_1}({  \R}^+), \,\,\,\,n_t(x,0)=n_1(x)\in \hat H^{s_1-1}({ \R}^+),\\
u(0,t)=h(t)\in H^{\frac{2s_0+1}{4}}(\R^+),\,\,\,n(0,t)=f(t)\in H^{s_1}(\R^+),
\end{array}
\right.
\end{equation}
with the additional compatibility conditions $h(0)=g(0)$ when $s_0>\frac12$, $f(0)=n_0(0) $ when $s_1>\frac12$.
The compatibility conditions are necessary since the solutions we are interested in are    continuous space-time functions for $s>\frac12$. The function $u(x,t)$ denotes the slowly varying envelope of the electric field with a prescribed frequency and the real valued function $n(x,t)$ denotes the deviation of the ion density from the equilibrium.  

Smooth solutions of the Zakharov system posed on $\R $ or $\T$ obey the following conservation laws:
$$\|u(t)\|_{L^2(\Bbb T)}=\|u_{0}\|_{L^2(\Bbb T)}$$
and
$$E(u,n,\nu)(t)= \int_{\Bbb T}|\partial_{x}u|^2dx+\frac{1}{2}\int_{\Bbb T} n^2dx+\frac{1}{2}\int_{\Bbb T}\nu^2dx+\int_{\Bbb T}n|u|^2dx=E(u_{0},n_{0},n_{1}),$$
where $\nu$ is such that $n_t=\nu_{x}$ and $\nu_t=(n+|u|^2)_x$. These conservation laws identify $H^1 \times L^2 \times \hat H^{-1}$ as the energy space for the system. Here $\hat H^{-1}=\{\partial_x f:f\in L^{2}\}.$ 
In the case of the semi infinite strip the solution $u,n$ models waves that are generated at one end and propagate freely at the other. In this paper we continue our program initiated in  \cite{et4} of  establishing the regularity properties of nonlinear dispersive partial differential equations (PDE) on a half line  using the tools that are available in the case of the real line where the PDE are fully dispersive. To this end we extend the data into the whole line and use Laplace transform methods to   set up an equivalent integral equation (on $\R\times \R$) of the solution, see \eqref{eq:duhamel} below.  We analyze the integral equation  using  the restricted norm method and multilinear $L^2$ convolution estimates.

For the initial and boundary value problem and for nonzero boundary data as far as we know the wellposedness theory of the Zakharov system is unknown. In this paper we adapt the method of \cite{et4} to the Zakharov system to establish the wellposedness theory and essentially match the known results on $\R$.
We say $(s_0,s_1)$ is admissible if  $s_0\in (0,\frac52)\backslash \{\frac12,\frac32\}$, $s_1\in (-\frac12,\frac32)\backslash\{\frac12\}$ satisfy
$$ 0<s_0-s_1\leq 1, \,\,\,\,\,2s_0>s_1+\frac12>0.$$ 

\begin{definition} 
We say \eqref{eq:zakharov} is locally wellposed in $H^{s_0}(\R^+)\times H^{s_1}(\R^+)$, if for any $g\in H^{s_0}(\R^+)$,  $h\in H^{\frac{2s_0+1}{4}}(\R^+)$,  
$n_0\in H^{s_1}({  \R}^+)$, $n_1\in\hat H^{s_1-1}({ \R}^+)$, and $f\in H^{s_1}(\R^+)$,
  with the additional compatibility conditions mentioned above, the integral equation \eqref{eq:duhamel} below has a unique solution in
\be\label{def:lwpspace}
\big[X^{s_0,b}  \cap C^0_tH^{s_0}_x  \cap C^0_xH^{\frac{2_0+1}{4}}_t \big] \times \big[Y^{s_1,b} \cap C^0_tH^{s_1}_x  \cap C^0_xH^{s_1}_t \big],
\ee
for $b<\frac12$ and  for sufficiently small $T$ depending only on the norms of the 
  boundary and initial data.    
 Furthermore,   the solution depends continuously on the initial and boundary data. In the critical case, $s_0=s_1+1$, we replace the space $X^{s_0,b}$ with $X^{s_0,\frac12}$ and $Y^{s_1,b}$ with $Y^{s_1,\frac12+}$. For the definition of the spaces $X^{s_0,b}$ and $Y^{s_1,b}$ see the next section.
\end{definition}
 
\begin{rmk}
The different approach in the critical $s_0=s_1+1$, and non-critical case for the Zakharov system is by now a standard feature of the theory. In the critical case one is forced to work in $X^{s_0,\frac12}$ for the Schr\"odinger part of the solution. To prove continuity of the flow one needs an  estimate on a new norm, see \cite{GTV}. This is the $Z$ norm in Proposition \ref{prop:critical}. We also loose a small power in the time variable but we can close the iteration scheme by an additional gain coming from the nonlinear estimates. The situation is similar to the real line case and more details can be found in \cite{GTV}.
\end{rmk}
Below we provide a short summary of results on the wellposedness theory of the Zakharov system. A more comprehensive account is given in \cite{et3} and in the references therein.
Recall that the Zakharov system is not scale invariant but it can be reduced to a simplified system like in \cite{GTV}, and one can then define a critical regularity. This is given by the pair $(s_{0},s_{1})=(\frac{d-3}{2},\frac{d-4}{2})$, which is   on the line $s_0=s_1+\frac12$.  

The restricted norm method was used to study the Zakharov system first in \cite{bc}.
Later in  \cite{GTV}  a comprehensive account of the local wellposedness theory  on $\R^d$ was given.   In dimensions $1$ and $2$, the lowest regularity for the system to have local solutions has been found to be $(s_{0},s_{1})=(0,-\frac{1}{2})$, \cite{GTV}, \cite{holmer}. It is harder to establish the global 
 solutions at  this level since there is no conservation law controling the wave part. On $\R$, this has been done in \cite{cht}.  This result is to be expected since the cubic NLS, which is the subsonic limit of the Zakharov system, is globally wellposed in one dimension. On $\R^2$, the cubic NLS is $L^2$ critical and the time of local existence depends not only on the norm of the initial data but also on their profile. In \cite{bhht} the authors prove $L^2$ wellposedness of the Zakharov system on $\R^2$. In addition they proved that the time of existence depends only on the norm of the initial data. In this sense the Zakharov system behaves better than its limit. However,  the local solutions cannot be extended globally-in-time, since we know that the $L^2$ solutions of the Zakharov system in two dimensions blow up, \cite{gmer,gmer1}.  On the other hand if the Schr\"odinger initial data have small $L^2$ norm, one can prove global wellposedness even with infinite energy data. For the details see \cite{fpz} and \cite{kishi1}.

For the periodic problem,  Bourgain proved that the problem is locally wellposed in the energy space using the restricted norm method, see \cite{jbz} and \cite{Bbook}.  Bourgain's result was extended in \cite{ht} to a  local wellposedness result in $H^{s_{0}}\times H^{s_{1}}\times H^{s_{1}-1}$ for  $s_1\geq 0$ and $\max(s_1, \frac{s_1}2+\frac12)\leq s_0\leq s_1+1 $.  A recent result, \cite{kishi}, establishes wellposedness in the case of the higher dimensional torus. The problem with Dirichlet boundary conditions has been considered in \cite{fla} and \cite{gm} in more regular spaces than the energy space.

The energy solutions exist both in $\Bbb R$ and $\Bbb T$ for all times due to the a priori bounds on the local theory norms. We should note that although the quantity $\int_{\R}n|u|^2dx$  has no definite sign it can be controlled using Sobolev inequalities by the $H^1$ norm of $u$ and the $L^2$ norm of $n$. This gives the following a priori bound, c.f. \cite{pecher},
\be\label{energybound}
\|u(t)\|_{H^1}+\|n(t)\|_{L^2}+\|n_t(t)\|_{H^{-1}}\les \|u(0)\|_{H^1}+\|n(0)\|_{L^2}+\|n_t(0)\|_{H^{-1}},\,\,\,\,\,\,t\in\R.
\ee
Unfortunately this is not the case for the equation \eqref{eq:zakharov} due to the presence of boundary conditions.

The main result  of this paper is the following theorem. The operators  $W_0^t$ and $V_0^t$ are the $u$ and $n$ parts of the solutions of the linear part of the system \eqref{eq:zakharov}, see Section 2 below.
\begin{theorem} \label{thm:local} For any admissible pair $(s_0,s_1)$ the equation \eqref{eq:zakharov}  is locally wellposed in    $H^{s_0}(\R^+)\times H^{s_1}(\R^+)$. Moreover, in the noncritical case when $s_0<s_1+1$,  we have the following smoothing bound
\begin{align*}
&u-W_0^t(g,h)\in C^0_tH^{s_0+a_0}([0,T]\times \R^+) \\
&n-V_0^t(n_0,n_1,f)\in C^0_tH^{s_1+a_1}([0,T]\times \R^+),
\end{align*}
for any $a_0<  \min\left(\frac12, s_1+\frac12,s_1-s_0+1, \frac52-s_0\right)$  and $a_1 < \min\left(s_0-s_1,2s_0-s_1-\frac12,\frac32-s_1\right)$. In the case $s_0=s_1+1$, we have the one sided smoothing
$$n-V_0^t(n_0,n_1,f)\in C^0_tH^{s_1+a_1}([0,T]\times \R),$$
for any $a_1< \min\left(1, \frac32-s_1\right)$.
\end{theorem}
  
\begin{rmk} For the Zakharov on the torus a similar smoothing estimate was proved in \cite{et3}. Following the method there one can prove similar estimates for the Zakharov system on the real line. A global smoothing estimate on $\R$ for data close to the energy space was given in \cite{pecher}. More recently smoothing on $\R^d$ was obtained in \cite{compaan}. 
\end{rmk}
To prove the above theorems we rely on the Duhamel formula of the nonlinear system adapted to the boundary conditions,  which expresses the nonlinear solution as the superposition of the linear evolutions which incorporate the boundary and the initial data with the nonlinearity. Thus, we first solve two linear problems \eqref{linearnls} and \eqref{eq:zak_lin} by a combination of Fourier and Laplace transforms, \cite{et4}, \cite{bonaetal}, after extending the inital data to the whole line. The idea is then to use the restricted norm method as in \cite{collianderkenig}, \cite{holmer1}, \cite{et4} in the Duhamel formula. The uniqueness of the solutions thus constructed is not immediate since we do not know that the fixed points of the Duhamel operators have restrictions on the half line which are independent of the extension of the data. To accomplish that we need some a priori bounds on the solutions. Such a priori estimates for the difference of two solutions and smooth data is a classical result, see \cite{aa1}, \cite{aa2}, \cite{ss}, \cite{GZZ}. To prove uniqueness for rougher data we use an approximation argument that relies on the existence of global smooth solutions. Unfortunately,  
the local conservation identities  leads to useful a priori global bounds only when $h=0$. Under this additional assumption, we obtain  global energy solutions and uniqueness in the full range.   We summarize our results in the following two theorems:

\begin{theorem} \label{thm:unique} For any admissible pair $(s_0,s_1)$ satisfying $s_0\geq 2$, $s_1\geq 1$, the solution of \eqref{eq:zakharov} given by Theorem~\ref{thm:local} is independent of the extensions of the initial data. The same statement also holds for the remaining admissible indices provided that $h=0$.  
\end{theorem}

\begin{theorem} \label{thm:global} For any admissible pair $(s_0,s_1)$ satisfying $s_0\geq 1$, $s_1\geq 0$, the solution of \eqref{eq:zakharov} given by Theorem~\ref{thm:local} can be extended to a solution   in 
$$
\left[C^0_{t\in[0,T]}H^{s_0}_{x\in\R^+} \cap C^0_{x\in\R^+} H^{\frac{2s_0+1}{4}}_{t\in[0,T]}\right] \times \left[C^0_{t\in[0,T]}H^{s_1}_{x\in\R^+}  \cap C^1_{t\in[0,T]} \hat H^{s_1-1}_{x\in\R^+} \cap  C^0_{x\in\R^+} H^{s_1}_{t\in[0,T]}\right]  
$$
 for any $T>0$ provided that  $h=0$. Furthermore, the smoothing given by Theorem~\ref{thm:local} remains valid, and $\|u\|_{H^{s_0}(\R^+)}$, $\|n\|_{H^{s_1}(\R^+)}$, $\|n_t\|_{\hat H^{s_1-1}(\R^+)}$ can grow at most polynomially-in-time.
\end{theorem}

We now discuss briefly the organization of the paper.  In Section 2,   we construct the solutions of the linear problems and set up the Duhamel formulas of the full system. The Duhamel formula incorporates the extension of the data on $\Bbb R$ and the evaluation of certain operators at the zero boundary. We note that the solution is constructed on $\Bbb R$ but its restriction on $\R^+$ satisfies the PDE in an appropriate sense.

In Section 3 we introduce the restricted norm method and prove the linear and nonlinear estimates (in the non-critical case) for the components of the solution that were defined in Section 2.   The estimates for the critical case are given in Section 4. Section 5 contains the description of the iteration that leads to the proof of Theorem \ref{thm:local}. We also provide the proof of uniqueness in the case of smooth solutions. In Section 6 under the assumption that the boundary term of the Schr\"odinger part of the solution is zero we obtain a priori bounds on the energy space and thus we show how the local solutions can be extended to global ones with initial data on the energy space or smoother. This result is then used in Section 7 where uniqueness in the case of rough initial data is proved by an approximation argument. The last section, Section 8, is an Appendix where we state two calculus lemmas that we use throughout the paper.

We finish this section by introducing some notation. We also define the solution spaces that give meaning to the definition of wellposedness.

\subsection{Notation}

\begin{align*}
\widehat g(\xi)&=\mathcal F g(\xi)=\int_{\R^n} e^{-ix\cdot\xi} g(x) dx.\\
\la \xi\ra&=\sqrt{1+|\xi|^2}.\\
\|g\|_{H^s}&=\|g\|_{H^s(\R)}=\Big(\int_\R \la \xi\ra^{2s} |\widehat g(\xi)|^2 d\xi\Big)^{1/2}.
\end{align*}

For   $s>-\frac12$, we define   the space  $$H^s(\R^+) =
\{g\in \mathcal D(\R^+): \exists \tilde g \in H^s(\R) \text{ so that } \tilde g \chi_{(0,\infty)} =g \},
$$
with the norm
$$
\|g\|_{H^s(\R^+)}:=\inf\big\{\|\tilde g\|_{H^s(\R)}:  \tilde g \chi_{(0,\infty)} =g \big\}.
$$
The restriction $s>-1/2$ is necessary since multiplication with $\chi_{(0,\infty)}$ is not well-defined for $s\leq -1/2$.  
 
For $s>-3/2$,  we define $\hat H^{s}(\R^+)=\{\partial_x f:f\in H^{s+1}(\R^+)\} $ with the norm   
$$
\|\partial_x f\|_{\hat H^{s}(\R^+)}=\|f\|_{H^{s+1}(\R^+)}. 
$$ 
We also define $\hat H^s(\R)$ analogously.

We denote the linear Schr\"odinger propagator (for $g\in L^2(\R)$) by
$$
e^{it\partial_{xx}} g(x)= \mathcal F^{-1}\big[e^{-it|\cdot|^2} \widehat g(\cdot)\big](x).
$$
We will similarly use the notation $e^{\pm t\partial_x}$ for the linear wave propagators on the whole line.

For a space time function $f$, we denote
$$
D_0f(t)=f(0,t).
$$
Finally, we  reserve the notation $\eta(t)$ for a smooth compactly supported function which is equal to $1$ on $[-1,1]$. 

\section{Notion of a solution} \label{sec:defin}
 
We start with some remarks on $H^s(\R^+)$ and $\hat H^s(\R^+)$ spaces. 
First, given $g  \in  H^s(\R^+)$ for some $s>\frac12$, take an extension $\tilde g\in H^s(\R)$. By Sobolev embedding    $\tilde g $ is continuous on $\R$, and hence $g(0)$ is well defined.   

Second, any $H^{s+1}(\R)$  extension $\tilde f$ of $f\in H^{s+1}(\R^+)$ gives an extension of $g=\partial_x f\in \hat H^s(\R^+)$ to $\R$ by
$$
\tilde g=\partial_x \tilde f\in \hat H^s(\R).
$$
Note that $\tilde g $ and $g$ agree as distributions on $\R^+$.
 
We have the following lemma from \cite{et4} (also see \cite{collianderkenig}) concerning extensions of $H^s(\R^+)$ functions. 
\begin{lemma}\label{lem:Hs0} Let $h\in H^s(\R^+)$ for some $-\frac12< s<\frac52$. \\
i) If $-\frac12< s<\frac12$, then $\|\chi_{(0,\infty)}h\|_{H^s(\R)}\les \|h\|_{H^s(\R^+)}$.\\
ii) If $\frac12<s<\frac32$  and $h(0)=0$, then $\|\chi_{(0,\infty)}h\|_{H^s(\R)}\les \|h\|_{H^s(\R^+)}$. 
\end{lemma}

To construct the solutions of \eqref{eq:zakharov} we first consider the linear Schr\"odinger and linear wave equations on $\R^+$ with boundary data. 
For the linear Schr\"odinger,
\be \label{linearnls}\left\{\begin{split}
&iu_t+u_{xx} =0,\,\,\,\,x\in\R^+, t\in \R^+,\\
&u(x,0)=g(x)\in H^s(\R^+), \,\,\,\,u(0,t)=h(t)\in H^{\frac{2s+1}4}(\R^+),  
\end{split}\right.\ee
with the compatibility condition $h(0)=g(0)$ for $s>\frac12$, we refer the reader to \cite{et4}, also see \cite{bonaetal}.  Note that the uniqueness of the solutions of equation \eqref{linearnls} follows  by considering the equation with $g=h=0$ with the method of odd extension. Given extension $g_e$ of $g$,  we denote the solution by  $W_{0}^t(g_e,h)$, which can be written as
$$
W_0^t(g_e,h)=W_0^t(0,h-p)+e^{it\partial_{xx}} g_e,
$$
where $p(t)= \eta(t) [e^{it\partial_{xx}}  g_e]\big|_{x=0}$, which is well-defined and is in  $H^{\frac{2s+1}{4}}(\R^+)$ by Lemma~\ref{lem:kato} below.
The properties of the boundary operator 
$W_0^t(0,h)$ were established in \cite{et4}. We state the results we need in the next section.
It is important to recall that $W_0^t(g_e,h)$ is defined on $\R\times \R$ but satisfies \eqref{linearnls} on $\R^+\times [0,1]$. Moreover its restriction to $\R^+$  is independent of the extension $g_e$.

We now consider the  linear wave problem:
\begin{equation}\label{eq:zak_lin}
\left\{
\begin{array}{l}
n_{tt}-n_{xx} =0,\,\,\,\,x,t\in\R^+, \\
n(x,0)=n_0(x) \in H^{s_1}(\R^+),\,\,\,\,\, n_t(x,0)=n_1(x)\in \hat H^{s_1-1}(\R^+),\\
n(0,t)=f(t)\in H^{s_1}(\R^+).
\end{array}
\right.
\end{equation}
First note that the restriction of 
\be\label{def:V0t}
 V_0^t(0, f)(x):= [\chi_{(0,\infty)}f](t-x)
 \ee
 to $\R^+$  is the solution of \eqref{eq:zak_lin} when $n_0=n_1=0$. 

As above, let $n_{e0}$ and $n_{e1}$ be extensions of $n_0$ and $n_1$ to $\R$ with the property that 
$$
\|n_{e0}\|_{H^{s_1}(\R)}\les \|n_0\|_{H^{s_1}(\R^+)}, \,\,\,\,\,\|n_{e1}\|_{\hat H^{s_1-1}(\R)}\les \|n_1\|_{\hat H^{s_1-1}(\R^+)}.
$$
We define
$$
\psi_{\pm}(x)=n_{e0}(x)\pm \partial_x^{-1} n_{e1}(x) \in H^{s_1}(\R),
$$
Let $V_0^t(\psi^{\pm},f)$ be defined on $\R\times [0,1]$ by 
\be\label{def:V1t}
V_0^t(\psi^{\pm}, f)(x)  =  \frac12  \left[e^{  t \partial_x} \psi_{+}+ e^{- t \partial_x} \psi_{-}\right] +  V_0^t(0, f-r)(x)
\ee
where 
\be\label{def:r}
r(t)= \frac12 \eta(t)\left(D_0 e^{   t \partial_x} \psi_{+} + D_0 e^{-   t \partial_x} \psi_{-}\right). 
\ee
Note that  the restriction of $ V_0^t(\psi^{\pm}, f) $ to $\R^+\times[0,1]$ is the solution of  \eqref{eq:zak_lin}.

Now we write a system of integral equations equivalent to  \eqref{eq:zakharov} on $[0,T]$, $T<1$:
\begin{equation}\label{eq:duhamel} \left\{
\begin{array}{l}
u(t)= \eta(t) W_0^t\big(g_e, h  \big) -i \eta(t) \int_0^t e^{i(t- t^\prime)\partial_{xx}}   F(u,n) \,d t^\prime  +i\eta(t) W_0^t\big(0,  q  \big), \\
n(t)= \eta(t) V_0^t\big(\psi^{\pm}, f\big) + \frac12\eta(t) (n_+ +  n_-) -\frac12 \eta(t) V_0^t(0,z), 
\end{array}\right.
\end{equation}
where
\be \label{def:F}
F(u,n)=\eta(t/T) n u,\,\,\,\, \text{ and }\,\, q(t)=  \eta(t ) D_0\Big(\int_0^t e^{i(t-t^\prime)\partial_{xx}}  F(u,n)\, d t^\prime \Big).
\ee
\be\label{def:npm}
n_\pm =   \pm  \int_0^t e^{\pm  (t-t^\prime) \partial_x} G(u )dt^\prime,\,\,\,\,G(u ):=\eta(t/T)   \partial_x  |u|^2, 
\ee
\begin{multline}\label{def:z}
z(t)=  \eta(t) D_0(n_++n_-)  \\=   \eta(t)\left(D_0\Big(\int_0^t e^{   (t- t^\prime)\partial_x} G(u ) dt^\prime \Big)-D_0\Big(\int_0^t e^{- (t- t^\prime)\partial_x} G(u ) dt^\prime \Big) \right).
\end{multline}

In what follows we will prove that the  integral equation \eqref{eq:duhamel} has a unique solution in a suitable Banach space on $\R\times \R$ for some $T<1$. Using the definition of the boundary operators, it is clear that the restriction of $u$ to $\R^+\times [0,T]$ satisfies  \eqref{eq:zakharov} in the distributional sense. Also note that the smooth solutions of \eqref{eq:duhamel} satisfy \eqref{eq:zakharov} in the classical sense.

We work with the spaces $X^{s,b}(\R\times\R)$  \cite{bourgain,Bou2}:
\be\label{def:xsb}
\|u\|_{X^{s,b}}=\big\| \widehat{u}(\tau,\xi)\la \xi\ra^{s} \la \tau+\xi^2\ra^{b} \big\|_{L^2_\tau L^2_\xi}=\big\|e^{-it\partial_{xx}}u\big\|_{H^s_xH^b_t},
\ee
we also define the wave equivalent of the  $X^{s,b}$ space:
\be \label{def:ysb}
\|n\|_{Y^{s,b}}=\inf_{n=n_+ + n_-} \left(\|n_+\|_{Y_+^{s,b}}+ \|n_-\|_{Y_-^{s,b}} \right),
\ee
where 
$$
\|n \|_{Y_\pm^{s,b}} = \big\| \widehat{n }(\tau,\xi)\la \xi\ra^{s} \la \tau\mp  \xi \ra^{b} \big\|_{L^2_\tau L^2_\xi}=\big\| e^{\mp t\partial_x} n  \big\|_{H^s_xH^b_t}.
$$
 
We  recall the embedding $X^{s,b}, Y^{s,b}\subset C^0_t H^{s}_x$ for $b>\frac{1}{2}$ and the following inequalities from \cite{bourgain,GTV}. 
 
For any $s,b\in \R$ we have
\begin{equation}\label{eq:xs1}
\|\eta(t) e^{it\partial_{xx}} g\|_{X^{s,b}}\les \|g\|_{H^s}.
\end{equation}
For any $s\in \mathbb R$,  $0\leq b_1<\frac12$, and $0\leq b_2\leq 1-b_1$, we have
\begin{equation}\label{eq:xs2}
\Big\| \eta(t) \int_0^t  e^{i(t-t^\prime)\partial_{xx}}  F(t^\prime ) dt^\prime \Big\|_{X^{s,b_2} }\lesssim   \|F\|_{X^{s,-b_1} }.
\end{equation}
Moreover, for $T<1$, and $-\frac12<b_1<b_2<\frac12$, we have
\begin{equation}\label{eq:xs3}
\|\eta(t/T) F \|_{X^{s,b_1}}\les T^{b_2-b_1} \|F\|_{X^{s,b_2}}.
\end{equation}
Also note that, in \eqref{eq:xs3}, one can replace strict inequalities with equalities in the $b$ indices by loosing $T^{-\epsilon}$, for a proof see \cite{GTV}.
Analogous inequalities hold for the norms $Y^{s,b}_\pm$.

\section{A priori estimates} \label{sec:apriori}

\subsection{Estimates for linear terms}

We start with the following estimates from \cite{et4} for the Schr\"odinger part of the equation:
\begin{lemma}\label{lem:kato}(Kato smoothing inequality) Fix $s\geq 0$. For any $g\in H^s(\R)$, we have
$\eta(t) e^{it\partial_{xx}}  g\in C^0_xH^{\frac{2s+1}4}_t(\R\times \R)$, and we have
$$
\big\|\eta e^{it\partial_{xx}}  g \big\|_{L^\infty_x H^{\frac{2s+1}{4}}_t }\les \|g\|_{H^s(\R)}.
$$
\end{lemma}
Lemma~\ref{lem:wbcont} and Proposition~\ref{prop:wbh} below show that the Schr\"odinger part of the boundary operator belongs to the space \eqref{def:lwpspace}.  
\begin{lemma}\label{lem:wbcont}  Let  $s\geq 0 $. Then for  $h$ satisfying  $\chi_{(0,\infty)}h\in H^{\frac{2s+1}{4}}(\R )$,   we have
$W_0^t(0, h) \in C^0_tH^s_x(\R\times \R)$, and $\eta(t)W_0^t(0, h) \in C^0_xH^{\frac{2s+1}4}_t(\R\times \R)$.
\end{lemma}

\begin{proposition} \label{prop:wbh} Let $b\leq \frac12$ and $s\geq 0 $. Then for  $h$ satisfying  $\chi_{(0,\infty)}h\in H^{\frac{2s+1}{4}}(\R )$,   we have
$$\|\eta(t) W_0^t(0,h) \|_{X^{s,b}} \les \|\chi_{(0,\infty)}h\|_{H_t^{\frac{2s+1}4}(\R )}.$$
\end{proposition}

Now, we present analogous results for the wave part of the equation.

\begin{lemma}\label{lem:katoW}  Fix $s\in \R$. For any $g\in H^s(\R)$, we have
$\eta(t) e^{\pm   t \partial_x } g\in C^0_xH^s_t(\R\times \R)$, and we have
$$
\big\|\eta   e^{\pm   t \partial_x} g \big\|_{L^\infty_x H^{s}_t }\les \|g\|_{H^s(\R)}.
$$
\end{lemma}
\begin{proof} This is immediate since  translation is a continuous operator on   Sobolev spaces.
\end{proof}
Proposition~\ref{prop:wbhW} and  Lemma~\ref{lem:wbcontW}  below show that the wave part of the boundary operator belongs to the space \eqref{def:lwpspace}.  
\begin{proposition} \label{prop:wbhW} Let $b\in\R$ and $s>-1/2 $. Then for  $h$ satisfying  $\chi_{(0,\infty)}f\in H^{s}(\R )$,   we have
$$\|\eta(t) V_0^t (0,f)  \|_{Y^{s,b}} \les \|\chi_{(0,\infty)}f\|_{H_t^{s}(\R )}.$$
\end{proposition}
 \begin{proof} Let $f_e(t):=\chi_{(0,\infty)}f\in H^s(\R)$. Note that the  definition \eqref{def:V0t}  of $V_0^t f$ is
$$
V_0^t(0, f)(x)=f_e(t-x).
$$
Therefore
$$
\|\eta(t) V_0^t (0,f) \|_{Y^{s,b}}\leq \|\eta(t) f_e(t-x) \|_{Y^{s,b}_-} =\|\eta(t) f_e(-x) \|_{H^s_xH^b_t} \les \|f_e\|_{H^s(\R)}. \qedhere
$$
\end{proof}

\begin{lemma}\label{lem:wbcontW}  Let  $s\in \R$. Then for  $f$ satisfying  $\chi_{(0,\infty)}f\in H^{s}(\R )$,   we have
$V_0^t(0,f) \in C^0_tH^s_x(\R\times \R)$, and $\eta(t)V_0^t(0,f) \in C^0_xH^{s}_t(\R\times \R)$.
\end{lemma}
\begin{proof} This is immediate from the definition of $V_0^t$ since  translation is a continuous operator on   Sobolev spaces.
\end{proof}

\subsection{Estimates for the nonlinear terms}
In this section we discuss estimates for the nonlinear terms in \eqref{eq:duhamel} in order to close the fixed point argument and to obtain the smoothing theorem. 
A variant of the following proposition was obtained in \cite{et4}, also see \cite{collianderkenig}. 

 \begin{proposition}\label{prop:duhamelkato}  For any $ b<\frac12$, we have
\begin{multline*}
\Big\|\eta  \int_0^te^{i(t- t^\prime)\partial_{xx}} F  dt^\prime  \Big\|_{C^0_xH^{\frac{2s+1}{4}}_t(\R\times \R)}\les \\ \left\{ \begin{array}{ll} \|F\|_{X^{s,-b}}& \text{ for } 0\leq s \leq \frac12,  \\
\|F\|_{X^{s,-b}}+ \big\| \int_{\R } \la \lambda+\xi^2\ra^{\frac{2s-3}{4}}   | \widehat F(\xi,\lambda)| d\xi  \big\|_{L^2_\lambda} & \text{ for }  \frac12 < s. \end{array}
\right.
\end{multline*}
\end{proposition}
\begin{proof}  
It suffices to prove the bound above for $\eta D_0\big(\int_0^t e^{i(t- t^\prime)\partial_{xx}} F  dt^\prime \big)$ since $X^{s,b}$ norm is independent of space translation. The continuity in $x$ follows from this.   Note that, ignoring the dimensional constants, 
$$
D_0\Big(\int_0^te^{i(t- t^\prime)\partial_{xx}} F  dt^\prime \Big)= \int_\R\int_0^t e^{-i(t-t^\prime)\xi^2}F(\widehat\xi,t^\prime) dt^\prime d\xi.
$$
Using
$$
F(\widehat\xi,t^\prime)=\int_\R e^{it^\prime\lambda}\widehat F(\xi,\lambda) d\lambda,
$$
and
$$
\int_0^t e^{it^\prime(\xi^2+\lambda)}dt^\prime = \frac{e^{it(\xi^2+\lambda)}-1}{i(\lambda+\xi^2)}
$$
we obtain
$$
D_0\Big(\int_0^t e^{i(t- t^\prime)\partial_{xx}} F  dt^\prime \Big) = \int_{\R^2} \frac{e^{it \lambda }-e^{-it\xi^2}}{i(\lambda+\xi^2)} \widehat F(\xi,\lambda) d\xi d\lambda.
$$
Let $\psi$ be a smooth cutoff for $[-1,1]$, and let $\psi^c=1-\psi$. We write
\begin{multline*}
 \eta(t) D_0\Big(\int_0^t e^{i(t- t^\prime)\partial_{xx}} F  dt^\prime \Big)=   \eta(t)  \int_{\R^2} \frac{e^{it \lambda }-e^{-it\xi^2}}{i(\lambda+\xi^2)} \psi(\lambda+\xi^2) \widehat F(\xi,\lambda) d\xi d\lambda \\ +\eta(t) \int_{\R^2} \frac{e^{it \lambda } }{i(\lambda+\xi^2)} \psi^c(\lambda+\xi^2) \widehat F(\xi,\lambda) d\xi d\lambda
-\eta(t) \int_{\R^2} \frac{ e^{-it\xi^2}}{i(\lambda+\xi^2)} \psi^c(\lambda+\xi^2) \widehat F(\xi,\lambda) d\xi d\lambda \\ =:I+II+III.
\end{multline*}
By Taylor expansion, we have
$$
  \frac{e^{it \lambda }-e^{-it\xi^2}}{ \lambda+\xi^2 } =-e^{it\lambda} \sum_{k=1}^\infty \frac{(-it)^k}{k!} (\lambda+\xi^2)^{k-1}
$$
Therefore, we have
\begin{multline*}
\|I\|_{H^{\frac{2s+1}{4}}(\R)}\les  \sum_{k=1}^\infty  \frac{\| \eta(t)t^k\|_{H^1}}{k!}  \Big\|  \int_{\R^2}  e^{it\lambda} (\lambda+\xi^2)^{k-1}  \psi(\lambda+\xi^2) \widehat F(\xi,\lambda) d\xi d\lambda\Big\|_{H_t^{\frac{2s+1}{4}}(\R)}\\
\les  \sum_{k=1}^\infty  \frac{1}{(k-1)!}  \Big\| \la \lambda\ra^{\frac{2s+1}{4}} \int_{\R }    (\lambda+\xi^2)^{k-1}  \psi(\lambda+\xi^2) \widehat F(\xi,\lambda) d\xi  \Big\|_{L^2_\lambda}\\
\les  \Big\| \la \lambda\ra^{\frac{2s+1}{4}} \int_{\R }       \psi(\lambda+\xi^2) |\widehat F(\xi,\lambda)| d\xi  \Big\|_{L^2_\lambda}.
\end{multline*}
By Cauchy-Schwarz inequality in $\xi$, we estimate this by
\begin{multline*}
\Big[  \int_{\R }  \la \lambda\ra^{\frac{2s+1}{2}} \Big(  \int_{|\lambda+\xi^2|<1} \la \xi\ra^{-2s} d\xi\Big)\Big(  \int_{|\lambda+\xi^2|<1}  \la \xi\ra^{2s} |\widehat F(\xi,\lambda)|^2 d\xi\Big) d\lambda \Big]^{1/2}\\
\les \|F\|_{X^{s,-b}} \sup_\lambda   \Big( \la \lambda\ra^{\frac{2s+1}{2}} \int_{|\lambda+\xi^2|<1} \la \xi\ra^{-2s} d\xi\Big)^{1/2}\les \|F\|_{X^{s,-b}} .
\end{multline*}
 The last inequality follows by a calculation substituting $\rho =\xi^2$.

To estimate $\|III\|_{H^{\frac{2s+1}{4}}(\R)}$, we divide the $\xi$ integral into two pieces, $|\xi|\geq 1$, $|\xi|<1$. We estimate the contribution of the former piece as above (after the change of variable $\rho=\xi^2$):
$$
   \Big\| \la \rho\ra^{\frac{2s+1}{4}}\int_{\R } \frac{1}{ \lambda+\rho } \psi^c(\lambda+\rho) \widehat F(\sqrt{\rho},\lambda) \frac{d\lambda}{\sqrt{\rho}}  \Big\|_{L^2_{|\rho|\geq 1}}
\les  \Big\| \la \rho\ra^{\frac{2s-1}{4}}\int_{\R } \frac{1}{ \la \lambda+\rho \ra } | \widehat F(\sqrt{\rho},\lambda)| \ d\lambda   \Big\|_{L^2_{|\rho|\geq 1}}.
$$
By Cauchy-Schwarz in $\lambda$ integral, and using $b<\frac12$, we bound this by
$$
   \Big[\int_{|\rho|>1}\int_{\R } \frac{ \la \rho\ra^{\frac{2s-1}{2}}}{ \la \lambda+\rho \ra^{2b} } | \widehat F(\sqrt{\rho},\lambda)|^2   d\lambda  d\rho  \Big]^{1/2}\les \|F\|_{X^{s,-b}}.
$$
We estimate the contribution of the latter term by
$$
 \int_{\R^2} \frac{ \|\eta(t) e^{-it\xi^2}\|_{H^{\frac{2s+1}{4}}}\chi_{[-1,1]}(\xi)}{|\lambda+\xi^2|} \psi^c(\lambda+\xi^2)  |\widehat F(\xi,\lambda)| d\xi d\lambda \les \int_{\R^2} \frac{   \chi_{[-1,1]}(\xi)}{\la\lambda+\xi^2\ra}    |\widehat F(\xi,\lambda)| d\xi d\lambda.
$$
For $b<\frac12$, this is bounded by $\|F\|_{X^{0,-b}}$ by Cauchy-Schwarz inequality in $\xi$ and $\lambda$ integrals.

For the second term, we have
\begin{multline*}
\|II\|_{H^{\frac{2s+1}{4}}(\R)}\les \|\eta\|_{H^1} \Big\| \la \lambda\ra^{\frac{2s+1}{4}}\int_{\R } \frac{1}{ \lambda+\xi^2 } \psi^c(\lambda+\xi^2) \widehat F(\xi,\lambda) d\xi  \Big\|_{L^2_\lambda}\\
\les  \Big\| \la \lambda\ra^{\frac{2s+1}{4}}\int_{\R } \frac{1}{ \la \lambda+\xi^2 \ra } | \widehat F(\xi,\lambda)| d\xi  \Big\|_{L^2_\lambda}
\\ \les  \Big\| \int_{\R } \la \lambda+\xi^2\ra^{\frac{2s-3}{4}} | \widehat F(\xi,\lambda)| d\xi  \Big\|_{L^2_\lambda} +    \Big\|  \int_{\R } \frac{|\xi|^{s+\frac12}}{ \la \lambda+\xi^2 \ra } | \widehat F(\xi,\lambda)| d\xi  \Big\|_{L^2_\lambda},
\end{multline*}
where we used the inequality $\la \lambda\ra \les \la \lambda+\xi^2\ra + \xi^2$. We bound the second summand by Cauchy-Schwarz as follows
\begin{multline*}
\Big[  \int_{\R }    \Big(  \int \frac{|\xi|}{\la \lambda+\xi^2 \ra^{2-2b}  } d\xi\Big)\Big(  \int \frac{ | \xi|^{2s}}{\la \lambda+\xi^2 \ra^{2b} } |\widehat F(\xi,\lambda)|^2 d\xi\Big) d\lambda \Big]^{1/2}\\
\les \|F\|_{X^{s,-b}} \sup_\lambda   \Big(    \int \frac{|\xi|}{\la \lambda+\xi^2 \ra^{2-2b} } d\xi\Big)^{1/2}\les \|F\|_{X^{s,-b}} .
\end{multline*}
This finishes the proof for $s>\frac12$. 
For $s\leq \frac12$,  by Cauchy-Schwarz inequality in $\xi$, we estimate 
$$
\Big\| \la \lambda\ra^{\frac{2s+1}{4}}\int_{\R } \frac{1}{ \la \lambda+\xi^2 \ra } | \widehat F(\xi,\lambda)| d\xi  \Big\|_{L^2_\lambda}$$
by
\begin{multline*}
\Big[  \int_{\R }  \la \lambda\ra^{\frac{2s+1}{2}} \Big(  \int \frac{1}{\la \lambda+\xi^2 \ra^{2-2b} \la \xi\ra^{2s}} d\xi\Big)\Big(  \int \frac{ \la \xi\ra^{2s}}{\la \lambda+\xi^2 \ra^{2b} } |\widehat F(\xi,\lambda)|^2 d\xi\Big) d\lambda \Big]^{1/2}\\
\les \|F\|_{X^{s,-b}} \sup_\lambda   \Big( \la \lambda\ra^{\frac{2s+1}{2}}  \int \frac{1}{\la \lambda+\xi^2 \ra^{2-2b} \la \xi\ra^{2s}} d\xi\Big)^{1/2}\les \|F\|_{X^{s,-b}} .
\end{multline*}
To obtain the last inequality recall that $s\leq\frac12, b<\frac12$, and  consider the cases $|\xi|<1$ and $|\xi|\geq 1$ separately. In the former case use $\la \lambda+\xi^2 \ra\sim\la\lambda\ra$, and in the latter case use Lemma~\ref{lem:sums} after the change of variable $\rho=\xi^2$. 
\end{proof}

We have a similar proposition for the wave part: 
\begin{proposition}\label{prop:duhamelkatoW}  For any $b<\frac12$, we have
\begin{multline*}
\Big\|\eta  \int_0^te^{\pm  (t- t^\prime)\partial_x} G dt^\prime  \Big\|_{C^0_xH^{s}_t(\R\times \R)} \les \\ \left\{ \begin{array}{ll}  \|G\|_{Y_\pm^{s,-b}} + \big\| \la \lambda\ra^s\int_{\R } \frac{1}{ \la \lambda\mp  \xi  \ra } | \widehat G(\xi,\lambda)| d\xi  \big\|_{L^2_\lambda}& \text{ for }  -\frac12 < s <0, \\
\|G\|_{Y_\pm^{s,-b}}& \text{ for } 0\leq s \leq \frac12, \\
 \|G\|_{Y_\pm^{s,-b}} +  \big\| \int_{\R }  \la \lambda\mp  \xi  \ra^{s-1}   | \widehat G(\xi,\lambda)| d\xi   \big\|_{L^2_\lambda}& \text{ for }  \frac12 < s. \end{array}
\right.
\end{multline*}
\end{proposition}
\begin{proof}  We consider only the '$+$' case.
As in Proposition~\ref{prop:duhamelkato}, it suffices to prove the bound above for $\eta D_0\big(\int_0^t e^{ (t- t^\prime)\partial_x} G  dt^\prime \big)$.
As above, we write
\begin{multline*}
 \eta(t) D_0\Big(\int_0^t e^{ (t- t^\prime)\partial_x} G  dt^\prime  \Big)=   \eta(t)  \int_{\R^2} \frac{e^{it \lambda }-e^{ it  \xi }}{i(\lambda-  \xi )} \psi(\lambda-  \xi ) \widehat G(\xi,\lambda) d\xi d\lambda \\ +\eta(t) \int_{\R^2} \frac{e^{it \lambda } }{i(\lambda- \xi )} \psi^c(\lambda-  \xi ) \widehat G(\xi,\lambda) d\xi d\lambda
-\eta(t) \int_{\R^2} \frac{ e^{ it  \xi }}{i(\lambda-  \xi )} \psi^c(\lambda-  \xi ) \widehat G(\xi,\lambda) d\xi d\lambda \\ =:I+II+III.
\end{multline*}
By Taylor expansion, we have
$$
  \frac{e^{it \lambda }-e^{ it  \xi }}{ \lambda-  \xi } = - e^{it\lambda} \sum_{k=1}^\infty \frac{(-it)^k}{k!} (\lambda- \xi  )^{k-1}
$$
Therefore,  we have
\begin{multline*}
\|I\|_{H^{s}(\R)}\les  \sum_{k=1}^\infty  \frac{\| \eta(t)t^k\|_{H^{s+1}}}{k!}  \Big\|  \int_{\R^2}  e^{it\lambda} (\lambda- \xi  )^{k-1}  \psi(\lambda-  \xi  ) \widehat G(\xi,\lambda) d\xi d\lambda\Big\|_{H_t^{s}(\R)}\\
\les  \sum_{k=1}^\infty  \frac{k^{s+1}}{k!}  \Big\| \la \lambda\ra^{s} \int_{\R }    (\lambda-  \xi )^{k-1}  \psi(\lambda-  \xi  ) \widehat G(\xi,\lambda) d\xi  \Big\|_{L^2_\lambda}\\
\les  \Big\| \la \lambda\ra^{s} \int_{\R }       \psi(\lambda-  \xi  ) |\widehat G(\xi,\lambda)| d\xi  \Big\|_{L^2_\lambda}.
\end{multline*}
By Cauchy-Schwarz inequality in $\xi$, we estimate this by
\begin{multline*}
\Big[  \int_{\R }  \la \lambda\ra^{ 2s } \Big(  \int_{|\lambda-\xi |<1} \la \xi\ra^{-2s} d\xi\Big)\Big(  \int_{|\lambda- \xi |<1}  \la \xi\ra^{2s} |\widehat G(\xi,\lambda)|^2 d\xi\Big) d\lambda \Big]^{1/2}\\
\les \|G\|_{Y_+^{s,-b}} \sup_\lambda   \Big( \la \lambda\ra^{2s} \int_{|\lambda- \xi  |<1} \la \xi\ra^{-2s} d\xi\Big)^{1/2}\les \|G\|_{Y_+^{s,-b}} .
\end{multline*}

To estimate $\|III\|_{H^{s}(\R)}$, we use Plancherel to obtain   
\begin{multline*}
 \|III\|_{H^s(\R)}\les   \Big\| \la \xi\ra^{s}\int_{\R } \frac{1}{ \la \lambda-\xi\ra } | \widehat G(\xi,\lambda) | d\lambda  \Big\|_{L^2 } \\
\les  \Big\| \la \xi\ra^{s}\Big( \int \frac{1}{ \la \lambda-\xi \ra^{2b} } | \widehat G(\xi,\lambda)|^2 d\lambda   \Big)^{1/2}\Big\|_{L^2 }\les \|G\|_{Y_+^{s,-b}}.   
\end{multline*}
In the second line we used Cauchy-Schwarz in $\lambda$ integral, and $b<\frac12$.

For the second term, we have
$$
\|II\|_{H^{s}(\R)}\les \Big\| \la \lambda\ra^s\int_{\R } \frac{1}{ \lambda-  \xi } \psi^c(\lambda-  \xi  ) \widehat G(\xi,\lambda) d\xi  \Big\|_{L^2_\lambda}\\
\les  \Big\| \la \lambda\ra^s\int_{\R } \frac{1}{ \la \lambda- \xi   \ra } | \widehat G(\xi,\lambda)| d\xi  \Big\|_{L^2_\lambda}.
$$
This finishes the proof for $-\frac12<s<0$.  For the remaining cases the proof is similar to the proof of Proposition~\ref{prop:duhamelkato}.
\end{proof}

Proposition~\ref{prop:smooth} and Proposition~\ref{prop:smooth12} below establish the smoothing for the nonlinear Schr\"odinger part:
\begin{proposition}\label{prop:smooth} For any admissible $s_0, s_1$ satisfying $s_0-s_1<1$, and for any  
$$a_0<\min\left(\frac12, s_1+\frac12,s_1-s_0+1\right),$$ 
there exists $\epsilon>0$ such that for $\frac12-\epsilon<b <\frac12$, we have
$$\big\| n u\big\|_{X^{s_0+a_0,-b }}\les \|n\|_{Y^{s_1,b}}\|u\|_{X^{s_0,b }}.$$ 
\end{proposition}
 \begin{proof}
It suffices to prove that   
$$\big\| n u\big\|_{X^{s_0+a_0,-b }}\les \|n\|_{Y^{s_1,b}_+}\|u\|_{X^{s_0,b }},$$
the proof involving the norm $Y^{s_1,b}_-$ is similar.

 By writing the Fourier transform of $ nu $ as a convolution, we obtain
\be\label{nuhat} \widehat{nu}(\xi,\tau) = \int  \widehat{n}(\xi_1,\tau_1)  \widehat{u}(\xi - \xi_1, \tau - \tau_1) d\xi_1 d\tau_1. 
\ee
Hence
\[
\| nu\|_{X^{s_0+a_0,-b }}^2 = \left\| \int  \frac{\langle\xi\rangle^{s_0+a_0} \widehat{n}(\xi_1,\tau_1) \widehat{u}(\xi - \xi_1, \tau - \tau_1)}{\langle \tau  + \xi^2 \rangle^{b }}d\xi_1 d\tau_1 \right\|_{L^2_\xi L^2_\tau}^2 . \]
We define
\[ f(\xi,\tau) = |\widehat{n}(\xi,\tau)|\langle \xi \rangle^{s_1} \langle \tau -  \xi \rangle^{b }, \]
\[ g(\xi,\tau) = |\widehat{u}(\xi,\tau)|\langle \xi \rangle^{s_0} \langle \tau+ \xi^2 \rangle^{b }, \]
and
\be\label{eq:MS}
 M(\xi_1, \xi,\tau_1, \tau) = \frac{ \langle\xi\rangle^{s_0+a_0} \langle \xi_1 \rangle^{-s_1}  \langle \xi -\xi_1 \rangle ^{-s_0}}{\langle \tau +\xi ^2 \rangle ^{b } \langle \tau_1 -  \xi_1  \rangle ^{b }  \langle \tau - \tau_1   + (\xi - \xi_1  )^2 \rangle ^{b } }. 
 \ee
It is then sufficient to show that
\begin{align*} \left\| \int  M(\xi_1, \xi,\tau_1, \tau) f(\xi_1,\tau_1) g(\xi-\xi_1 ,\tau-\tau_1 ) d\xi_1 d\tau_1 \right\|_{L^2_\xi L^2_\tau}^2 \\
\lesssim \|f\|_{L^2}^2 \|g\|_{L^2}^2= \|n\|_{Y^{s_1,b}_+}^2\|u\|_{X^{s_0,b }}^2.
\end{align*}
By applying Cauchy-Schwarz in the $\xi_1 \tau_1$ integral and then using H\"{o}lder's inequality, we bound the norm above by
\begin{align*}
&\left\| \left( \int M^2 d\xi_1 d\tau_1 \right)^{1/2}  \left( \int f^2(\xi_1,\tau_1)g^2(\xi-\xi_1 ,\tau-\tau_1 )  d\xi_1 d\tau_1 \right)^{1/2} \right\|_{L^2_\xi L^2_\tau}^2 \\
&\leq \sup_{\xi ,\tau} \left( \int  M^2d\xi_1 d\tau_1  \right)   \left\| \int  f^2(\xi_1,\tau_1)g^2(\xi-\xi_1 ,\tau-\tau_1 )  d\xi_1 d\tau_1  \right\|_{L^1_\xi L^1_\tau}\\
&= \sup_{\xi ,\tau}  \left( \int  M^2d\xi_1 d\tau_1  \right)    \|f\|_{L^2}^2 \|g\|_{L^2}^2.
\end{align*}
Thus it is sufficient to show that the supremum above is finite.
Using Lemma~\ref{lem:sums} in the $\tau_1 $ integral, the supremum is bounded by
\[ \sup_{\xi, \tau} \int \frac{\langle\xi\rangle^{2s_0+2a_0} \langle \xi_1 \rangle^{-2s_1}  \langle \xi -\xi_1 \rangle ^{-2s_0}}{\langle \tau + \xi ^2 \rangle ^{2 b } \langle \tau- \xi_1    + (\xi - \xi_1 )^2 \rangle ^{4b -1} } d\xi_1. \]

Using the relation $\langle \tau - a \rangle \langle \tau - b \rangle \gtrsim \langle a-b \rangle$, the  above reduces to
\begin{align*}& \quad \sup_{\xi}  \int \frac{\langle\xi\rangle^{2s_0+2a_0} \langle \xi_1 \rangle^{-2s_1}  \langle \xi -\xi_1 \rangle ^{-2s_0}}{   \langle \xi^2+  \xi_1  - (\xi - \xi_1 )^2 \rangle^{1-} } d\xi_1  \\
&= \sup_{\xi} \int \frac{\langle\xi\rangle^{2s_0+2a_0} \langle \xi_1 \rangle^{-2s_1}  \langle \xi -\xi_1 \rangle ^{-2s_0}}{   \big\langle  \xi_1   -\xi_1^2+2 \xi_1 \xi   \big\rangle^{1-} } d\xi_1.
\end{align*}
Replacing $\xi$ with $\xi-\frac12$, it suffices to estimate
$$
\sup_{\xi} \int \frac{\langle\xi\rangle^{2s_0+2a_0} \langle \xi_1 \rangle^{-2s_1}  \langle \xi -\xi_1 \rangle ^{-2s_0}}{   \big\langle  \xi_1^2-2 \xi_1 \xi  \big\rangle^{1-} } d\xi_1.
$$
Note that for $|\xi|\les 1$, the integral is bounded by 
$$
 \int \frac{  \langle \xi_1 \rangle^{-2s_1-2s_0} }{   \big\langle  \xi_1   \big\rangle^{2-} } d\xi_1 \les 1, 
$$
provided that $s_0+s_1 >-\frac12$.

For $|\xi|\gg 1$, we consider the cases: i) $|\xi_1| <1,$ ii) $|\xi_1-2\xi|<1$, and iii) $|\xi_1| >1$ and $|\xi_1-2\xi|>1$.

In the first case, we estimate the integral after the change of variable $\rho=\xi_1^2-2\xi \xi_1$, $d\rho=2(\xi_1-\xi) d\xi_1$:
$$\sup_{|\xi|\gg 1}\int_{|\rho|\les |\xi|} \frac{|\xi|^{ 2a_0-1}    }{   \langle  \rho \rangle^{1-}  } d\rho\les 1,
$$
provided that $a_0<\frac12$.

By the same change of variable we estimate the integral in the second case by
$$\sup_{|\xi|\gg 1}\int_{|\rho|\les |\xi|} \frac{|\xi|^{ 2a_0-2s_1-1}    }{   \langle  \rho \rangle^{1-}  } d\rho\les 1,
$$
provided that $a_0<s_1+\frac12$.

In the third case, we bound the integral by (after the change of variable $\xi_1\to \xi_1+\xi$)
$$
\sup_{\xi} \int \frac{\langle\xi\rangle^{2s_0+2a_0}}{  \la \xi_1 -   \xi   \ra^{1-}   \langle  \xi_1 \rangle ^{ 2s_0} \langle \xi_1+\xi \rangle^{ 2s_1+1-}  }  d\xi_1.
$$ 
Using Lemma~\ref{lem:sums2}, we bound this by
$$
\langle\xi\rangle^{2s_0+2a_0} \la\xi\ra^{-2s_1-2s_0-2+ \max(1,2s_1+1,2s_0)+}= \la\xi\ra^{2a_0-2s_1-2+  \max(1,2s_1+1,2s_0)+},
$$
which is bounded for 
$$a_0<s_1+1-\frac12\max(1,2s_1+1,2s_0)=\min(s_1+\frac12,s_1-s_0+1).   \qedhere$$
\end{proof}
 \begin{proposition}\label{prop:smooth12} For any admissible $s_0, s_1$ satisfying $s_0-s_1<1$, and for any  
$$ 
\frac12-s_0<a_0<\min\left(\frac12, s_1+\frac12,s_1-s_0+1,\frac52-s_0\right),
$$ 
there exists $\epsilon>0$ such that for $\frac12-\epsilon<b <\frac12$, we have
\be\label{L2normint}
\Big\| \int_{\R } \la \lambda+\xi^2 \ra^{\frac{2(s_0+a_0)-3}{4}}   | \widehat{nu}(\xi,\lambda)| d\xi  \Big\|_{L^2_\lambda} \les \|u\|_{X^{s_0,b}}  \|n\|_{Y^{s_1,b}}.
\ee
\end{proposition}
\begin{proof}
We give the proof when there is the norm $ \|n\|_{Y^{s_1,b}_+} $ in the right hand side of \eqref{L2normint}, the case of $ \|n\|_{Y^{s_1,b}_-} $ is similar.
Using \eqref{nuhat}, we can bound the left hand side of \eqref{L2normint} by
\begin{multline*}
\Big\|  \int_{\R^3 } \la \lambda+\xi^2 \ra^{\frac{2(s_0+a_0)-3}{4}}  |\widehat{n}(\xi_1,\lambda_1)|  |\widehat{u}(\xi - \xi_1, \lambda - \lambda_1)| d\xi_1 d\lambda_1  d\xi  \Big\|_{L^2_\lambda}\\
= \Big\|  \int_{\R^3 } \frac{\la \lambda+\xi^2 \ra^{\frac{2(s_0+a_0)-3}{4}}  f(\xi_1,\lambda_1) g (\xi - \xi_1, \lambda - \lambda_1) }{ \la \xi_1\ra^{s_1} \la \xi-\xi_1\ra^{s_0} \la \lambda_1- \xi_1 \ra^b \la \lambda-\lambda_1+(\xi-\xi_1)^2\ra^b  }   d\xi_1 d\lambda_1  d\xi  \Big\|_{L^2_\lambda},
\end{multline*}
where $f$ and $g$ are as in the proof of Proposition~\ref{prop:smooth}. Applying Cauchy-Schwarz inequality in $\xi_1,\lambda_1,\xi$ variables it suffices to prove that
$$
\sup_\lambda \int_{\R^3 } \frac{\la \lambda+\xi^2 \ra^{ s_0+a_0 -\frac32}  }{ \la \xi_1\ra^{2s_1} \la \xi-\xi_1\ra^{2s_0} \la \lambda_1- \xi_1 \ra^{2b} \la \lambda-\lambda_1+(\xi-\xi_1)^2\ra^{2b} }   d\xi_1 d\lambda_1  d\xi  <\infty.
$$ 
Using Lemma~\ref{lem:sums} in the $\lambda_1$ integral, we bound the supremum by
\be\label{sups2}
\sup_\lambda \int_{\R^2 } \frac{\la \lambda+\xi^2 \ra^{ s_0+a_0 -\frac32}  }{ \la \xi_1\ra^{2s_1} \la \xi-\xi_1\ra^{2s_0}   \la \lambda- \xi_1 +(\xi-\xi_1)^2\ra^{1-}  }   d\xi_1   d\xi.
\ee

Case  a) $\frac32\leq s_0+a_0<\frac52$. Using 
$$
\la \lambda+\xi^2\ra \les \la \lambda- \xi_1 +(\xi-\xi_1)^2\ra \la \xi_1\ra \la \xi_1-2\xi\ra
$$
we bound \eqref{sups2} by 
$$
 \int_{\R^2 } \frac{\la \xi_1-2\xi\ra^{ s_0+a_0 -\frac32}  }{ \la \xi_1\ra^{2s_1+\frac32-s_0-a_0} \la \xi-\xi_1\ra^{2s_0}     }   d\xi_1   d\xi.
$$
For $|\xi|\les |\xi_1|$, we bound this by
$$
 \int_{\R^2 } \frac{1 }{ \la \xi_1\ra^{2s_1+3-2s_0-2a_0} \la \xi-\xi_1\ra^{2s_0}     }   d\xi_1   d\xi \les 1,
$$ 
by integrating first in $\xi$ then in $\xi_1$  since $2s_0>1$ and $2s_1+3-2s_0-2a_0>1$. For $|\xi|\gg|\xi_1|$, we bound it by
$$
 \int_{|\xi| \gg|\xi_1| } \frac{1 }{ \la \xi_1\ra^{2s_1+\frac32-s_0-a_0} \la \xi \ra^{ s_0-a_0+\frac32}     }   d\xi  d\xi_1 \les \int  \frac{1 }{ \la \xi_1\ra^{2s_1+2-2a_0}      }    d\xi_1 \les 1,
$$
since $a_0<s_1+\frac12$.

Case  b) $\frac12 <  s_0+a_0<\frac32$. Using the inequality $\langle \tau - a \rangle \langle \tau - b \rangle \gtrsim \langle a-b \rangle$ and the change of variable $\xi\to \xi- \frac12$, we bound \eqref{sups2} by
$$
\int_{\R^2 } \frac{1  }{ \la \xi_1\ra^{2s_1} \la \xi-\xi_1\ra^{2s_0}   \la \xi_1(\xi_1-2\xi)\ra^{\frac32-s_0-a_0}  }   d\xi_1   d\xi.
$$
The resonant cases when $|\xi_1|\les 1$ or $|\xi_1-2\xi|\les 1$ can be handled as before. In the remaining case, we bound this integral by 
$$
\int_{\R^2 } \frac{1  }{ \la \xi_1\ra^{2s_1+\frac32-s_0-a_0} \la \xi-\xi_1\ra^{2s_0}   \la  \xi_1-2\xi \ra^{\frac32-s_0-a_0}  }   d\xi_1   d\xi.
$$
 Using Lemma~\ref{lem:sums2} in the $\xi_1$ integral we bound this by 
 $$
\int_{\R  } \frac{1  }{ \la \xi \ra^{2s_1+3 -2a_0 -\ell +}  }     d\xi,
$$
where 
$$
\ell=\max(1,2s_1+\frac32-s_0-a_0,2s_0,\frac32-s_0-a_0).
$$
We leave it to the reader to check that the integral above is finite under the hypothesis of the proposition, noting in addition that $s_0+s_1>0$.   
\end{proof}
 
Proposition~\ref{prop:smooth2}, Proposition~\ref{prop:smooth22}, and  Proposition~\ref{prop:smooth23} below establish the smoothing for the nonlinear wave part:
 \begin{proposition}\label{prop:smooth2} For any admissible $s_0, s_1$  and for any  
$$ 
a_1<\min\left(s_0-s_1,2s_0-s_1-\frac12\right).
$$ 
there exists $\epsilon>0$ such that for $\frac12-\epsilon<b <\frac12$, we have
 $$
 \| \partial_x  |u|^2 \|_{Y^{s_1+a_1,-b}_\pm} \les \|u\|_{X^{s_0,b}}^2.
 $$
\end{proposition}
\begin{proof}  
As in the proof of Proposition~\ref{prop:smooth}, it suffices to prove that (we only consider the norm $Y_+^{s_1+a_1,-b}$)
$$
\sup_{\xi ,\tau}  \left( \int  M^2d\xi_1 d\tau_1  \right)<\infty,
$$
where
\[ M(\xi_1, \xi,\tau_1, \tau) = \frac{ \langle\xi\rangle^{s_1+a_1}  |\xi| \langle \xi_1 \rangle^{-s_0}  \langle \xi -\xi_1 \rangle ^{-s_0}}{\langle \tau - \xi  \rangle ^{b } \langle \tau_1 -  \xi_1^2 \rangle ^{b }  \langle \tau - \tau_1   + (\xi - \xi_1  )^2 \rangle ^{b } }. \]
As before we bound the supremum by
\begin{multline*}
\sup_{\xi }    \int \frac{ \langle\xi\rangle^{2s_1+2a_1 }  |\xi|^2 \langle \xi_1 \rangle^{-2s_0}  \langle \xi -\xi_1 \rangle ^{-2s_0}}{  \langle  \xi  -\xi_1^2 + (\xi - \xi_1  )^2 \rangle ^{1-} }  d\xi_1 \\
\les \sup_{\xi }    \int \frac{ \langle\xi\rangle^{2s_1+2a_1 } |\xi|^2 \langle \xi_1 \rangle^{-2s_0}  \langle \xi -\xi_1 \rangle ^{-2s_0}}{  \langle \xi^2 -2\xi \xi_1 \rangle ^{1-} }  d\xi_1.
\end{multline*}
In the second line we made the change of variable $\xi_1\to\xi_1+ \frac12$.

For $|\xi|\les 1$, we estimate the supremum by 
$$
 \sup_{|\xi|\les 1 }    \int \frac{  |\xi|^2 }{  \langle \xi_1 \rangle^{ 4s_0}  \langle  \xi \xi_1 \rangle ^{1-} }  d\xi_1 \les 1,
$$
provided that $s_0>0$.

For $|\xi|\gg 1$ we have the cases: $|\xi-2\xi_1|<1$ and $|\xi-2\xi_1|>1$. In the former case we bound the integral by
$$
 \int_{|\xi-2\xi_1|<1} \frac{ \langle\xi\rangle^{2s_1+2a_1+2-4s_0 } }{  \langle \xi^2 -2\xi \xi_1 \rangle ^{1-} }  d\xi_1 = \int_{|\rho|<1} \frac{ \langle\xi\rangle^{2s_1+2a_1+2-4s_0 } }{  \langle \xi \rho \rangle ^{1-} }  d\rho \les \langle\xi\rangle^{2s_1+2a_1+1-4s_0 +},
$$
which is bounded provided that $a_1<2s_0 -s_1 -\frac12$. 

In the latter case we bound the integral by 

$$
 \int \frac{ \langle\xi\rangle^{2s_1+2a_1 +1+} }{  \langle \xi_1 \rangle^{ 2s_0}  \langle \xi -\xi_1 \rangle ^{ 2s_0} \langle  \xi- 2 \xi_1 \rangle ^{1-} }  d\xi_1.
$$
Using Lemma~\ref{lem:sums2}, we bound this by
$$
\langle\xi\rangle^{2s_1+2a_1+1+} \la \xi\ra^{-1-4s_0+\max(2s_0, 1)+}=\langle\xi\rangle^{2s_1+2a_1  -4s_0+\max(2s_0, 1)+},
$$
which is bounded provided that 
$$
a_1<\min(s_0-s_1,2s_0-s_1-\frac12).   \qedhere
$$
\end{proof}

 \begin{proposition}\label{prop:smooth22} For any admissible $s_0, s_1$  and for any  
$$ 
\frac12-s_1<a_1<\min\left(s_0-s_1,2s_0-s_1-\frac12,\frac32-s_1\right).
$$ 
there exists $\epsilon>0$ such that for $\frac12-\epsilon<b <\frac12$, we have
 $$
 \left\| \int_{\R }  \la \lambda\mp  \xi  \ra^{s_1+a_1-1}   |  \mathcal F( \partial_x  |u|^2 )(\xi,\lambda)| d\xi  \right\|_{L^2_\lambda}     \les \|u\|_{X^{s_0,b}}^2 .
 $$
\end{proposition}
\begin{proof} We provide the proof only for the '-' sign in the integral. We consider
$$
 \left\| \int_{\R }  \la \lambda-  \xi  \ra^{s_1+a_1-1}  |\xi|  |\widehat u( \xi_1, \lambda_1) |   |\widehat u(\xi-\xi_1,\lambda -\lambda_1)  |  d\xi_1d\lambda_1 d\xi  \right\|_{L^2_\lambda}.
$$
 
As in the proof of Proposition~\ref{prop:smooth12}, it suffices to prove that 
$$
\sup_\lambda \int_{\R^3 } \frac{\la \lambda- \xi   \ra^{2s_1+2a_1-2  } \,\,  |\xi|^2 }{ \la \xi_1\ra^{2s_0} \la \xi-\xi_1\ra^{2s_0} \la \lambda_1- \xi_1^2\ra^{2b} \la \lambda-\lambda_1+(\xi-\xi_1)^2\ra^{2b} }   d\xi_1 d\lambda_1  d\xi  <\infty.
$$ 
Using Lemma~\ref{lem:sums} in the $\lambda_1$ integral,  we bound the supremum by
\be\label{bothregions}  \sup_\lambda \int_{\R^2 } \frac{\la \lambda- \xi   \ra^{2s_1+2a_1-2  } \,\,  |\xi|^2 }{ \la \xi_1\ra^{2s_0} \la \xi-\xi_1\ra^{2s_0}   \la \lambda  +\xi  (\xi-2\xi_1) \ra^{1-} }   d\xi_1  d\xi.
\ee
By the change of variable $\rho:=\xi(\xi-2\xi_1)$ in the $\xi_1$ integral, we  bound the integral as
$$\sup_\lambda \int_{\R^2 } \frac{\la \lambda- \xi   \ra^{2s_1+2a_1-2  } \,\,  |\xi|  }{ \la \xi-\frac\rho\xi\ra^{2s_0} \la \xi+\frac\rho\xi\ra^{2s_0}   \la \lambda  +\rho \ra^{1-} }   d\rho   d\xi.
$$
Without loss of generality $|\xi-\frac\rho\xi|\gtrsim |\xi|$, which leads to the bound 
\be\label{regionii}
\sup_\lambda \int_{\R^2 } \frac{\la \lambda- \xi  \ra^{2s_1+2a_1-2  } \,\,  |\xi| \la \xi\ra^{-2s_0} }{ \la \xi+\frac\rho\xi\ra^{2s_0}   \la \lambda  +\rho \ra^{1-} }   d\rho   d\xi.
\ee
First we consider the case $1\leq s_1+a_1<\frac32$. Using $s_0>\frac12$, then integrating in $\rho$ we bound \eqref{regionii} by
$$
 \sup_\lambda \int_{\R^2 } \frac{\la \lambda- \xi   \ra^{2s_1+2a_1-2  } \,\,  |\xi|^2 \la \xi\ra^{-2s_0} }{  |\xi^2+\rho|^{1-}   \la \lambda  +\rho \ra^{1-} }   d\rho   d\xi\les \sup_\lambda \int_{\R } \frac{\la \lambda- \xi   \ra^{2s_1+2a_1-2  } \,\,  |\xi|^2 \la \xi\ra^{-2s_0} }{    \la \lambda -\xi^2 \ra^{1-} }     d\xi.
$$
Note that  $\la \lambda- \xi  \ra \les  \la \lambda -\xi^2 \ra + \la \xi\ra^2 $. Thus, we bound the integral by
$$
\sup_\lambda \int_{\R } \frac{  |\xi|^2 \la \xi\ra^{-2s_0+4s_1+4a_1-4 } }{    \la \lambda  -\xi^2 \ra^{1-} }     d\xi +\sup_\lambda \int_{\R } \frac{  |\xi|^2 \la \xi\ra^{-2s_0 } }{    \la \lambda  -\xi^2 \ra^{3-2s_1-2a_1-} }     d\xi .
$$
For $|\xi|\les 1$, these integrals are finite. For $|\xi|\gg 1$,   with the change of variable $\eta=\xi^2\pm \xi \gtrsim \xi^2 \gg 1 $, we bound these integrals by
$$  \sup_\lambda \int_{\R } \frac{ \la \eta\ra^{-\frac32+2s_1+2a_1- s_0} }{   \la \lambda-\eta \ra^{1-}    }     d\eta  +  \sup_\lambda \int_{\R } \frac{ \la \eta\ra^{\frac12- s_0} }{   \la \lambda-\eta \ra^{3-2s_1-2a_1-}    }     d\eta \les 1,
$$
provided that $-\frac32+2s_1+2a_1- s_0 <0$. This holds since $s_1+a_1<\frac32$ and $s_1+a_1<s_0$.

Now we consider the case $\frac12<s_1+a_1<1$.   Using $s_0>0$, we bound \eqref{regionii}  by 
$$
\sup_\lambda \int_{\R^2 } \frac{\la \lambda- \xi  \ra^{2s_1+2a_1-2  } \,\,  |\xi| \la \xi\ra^{-2s_0+} }{  |\xi^2+\rho|^{0+}   \la \lambda +\rho \ra^{1-} }   d\rho   d\xi\les \sup_{\lambda} \int_\R \frac{1 }{ \la \lambda- \xi   \ra^{2-2s_1-2a_1 }  \la \xi\ra^{2s_0-1-} }    d\xi \les 1,
$$
since $s_0>\frac12,$ $s_1+a_1<1$, and $a_1<s_0-s_1$.
\end{proof}
The following proposition is relevant only when $-\frac12<s_1+a_1<0$, which appears in the range of $a_1$ in the statement. This does not introduce an additional restriction on the parameter $a_1$ since the integral in the statement does not appear in other cases. 
 \begin{proposition}\label{prop:smooth23} For any admissible $s_0, s_1$  and for any  
$$ 
-\frac12-s_1<a_1<\min\Big(s_0-s_1,2s_0-s_1-\frac12, -s_1\Big).
$$ 
there exists $\epsilon>0$ such that for $\frac12-\epsilon<b <\frac12$, we have
 $$
 \left\| \la \lambda \ra^{s_1+a_1} \int_{\R } \frac{  |  \mathcal F( \partial_x  |u|^2 )(\xi,\lambda)|}{ \la \lambda\mp  \xi  \ra}  d\xi  \right\|_{L^2_\lambda}     \les \|u\|_{X^{s_0,b}}^2 .
 $$
\end{proposition}
\begin{proof} We provide the proof only for the '-' sign in the integral. We consider
$$
 \left\| \int_{\R }  \frac{\la \lambda  \ra^{s_1+a_1 }  |\xi|}{ \la \lambda - \xi   \ra}  |\widehat u( \xi_1, \lambda_1) |   |\widehat u(\xi-\xi_1,\lambda-\lambda_1)  |  d\xi_1d\lambda_1 d\xi  \right\|_{L^2_\lambda}
$$
 
As in the proof of Proposition~\ref{prop:smooth12}, it suffices to prove that 
$$
\sup_\lambda \int_{\R^3 } \frac{\la \lambda  \ra^{2s_1+2a_1  } \la \lambda- \xi   \ra^{ -2  } \,\,  |\xi|^2 }{ \la \xi_1\ra^{2s_0} \la \xi-\xi_1\ra^{2s_0} \la \lambda_1- \xi_1^2\ra^{2b} \la \lambda-\lambda_1+(\xi-\xi_1)^2\ra^{2b} }   d\xi_1 d\lambda_1  d\xi  <\infty.
$$ 
Using Lemma~\ref{lem:sums} in the $\lambda_1$ integral, we bound the supremum by
\be\label{bothregions2}  \sup_\lambda \int_{\R^2 } \frac{\la \lambda  \ra^{2s_1+2a_1  } \la \lambda- \xi   \ra^{ -2  } \,\,  |\xi|^2 }{ \la \xi_1\ra^{2s_0} \la \xi-\xi_1\ra^{2s_0}   \la \lambda  +\xi  (\xi-2\xi_1) \ra^{1-} }   d\xi_1  d\xi.
\ee
By the change of variable $\rho:=\xi(\xi-2\xi_1)$ in the $\xi_1$ integral, we  bound the integral as
$$\sup_\lambda \int_{\R^2 } \frac{\la \lambda  \ra^{2s_1+2a_1  }\la \lambda- \xi   \ra^{ -2  } \,\,  |\xi|  }{ \la \xi-\frac\rho\xi\ra^{2s_0} \la \xi+\frac\rho\xi\ra^{2s_0}   \la \lambda  +\rho \ra^{1-} }   d\rho   d\xi.
$$
Without loss of generality $|\xi-\frac\rho\xi|\gtrsim |\xi|$, which leads to the bound 
\begin{multline*}
\sup_\lambda \int_{\R^2 } \frac{\la \lambda  \ra^{2s_1+2a_1  } \la \lambda- \xi   \ra^{ -2  }\,\,  |\xi| \la \xi\ra^{-2s_0} }{ \la \xi+\frac\rho\xi\ra^{2s_0}   \la \lambda  +\rho \ra^{1-} }   d\rho   d\xi  \\ \les \sup_\lambda \int_{\R^2 } \frac{\la \lambda  \ra^{2s_1+2a_1  } \la \lambda -\xi  \ra^{ -2  }\,\,  |\xi|^{1+\min(2s_0,1) } \la \xi\ra^{-2s_0} }{ |\xi^2+\rho|^{\min(2s_0,1)-}   \la \lambda  +\rho \ra^{1-} }   d\rho   d\xi \\
\les \sup_\lambda \int_{\R } \frac{\la \lambda  \ra^{2s_1+2a_1  } \la \lambda- \xi  \ra^{ -2  }\,\,  |\xi|^{1+\min(2s_0,1) } \la \xi\ra^{-2s_0} }{    \la \lambda -\xi^2 \ra^{\min(2s_0,1)-  }}     d\xi.
\end{multline*}
 We have two cases $\la\lambda- \xi \ra \gtrsim \la \xi\ra^2 $ or $ \la\lambda -\xi^2\ra \gtrsim \la \xi\ra^2 $. In the former case we estimate the integral by
 $$
  \int_{\R }   |\xi|^{1+\min(2s_0,1) } \la \xi\ra^{-2s_0 -4 }    d\xi \les 1. 
 $$ 
 In the latter case, we have the bound 
\begin{multline*}
 \sup_\lambda \int_{\R } \la \lambda  \ra^{2s_1+2a_1  } \la \lambda- \xi   \ra^{ -2  }\,\,  \la \xi\ra^{1-2s_0- \min(2s_0,1)+}   d\xi \\\les
 \sup_\lambda \int_{\R }  \la \lambda- \xi   \ra^{ -2-2s_1-2a_1  }\,\,  \la \xi\ra^{2s_1+2a_1+1-2s_0- \min(2s_0,1)+}   d\xi \les 1
 \end{multline*}
 since $-\frac12<s_1+a_1<0$ and $a_1< \min\left(s_0-s_1,2s_0-s_1-\frac12 \right)$.
\end{proof}

\section{Estimates for the nonlinear terms in the critical case}  \label{sec:critical}
In this section we consider the case $s_0,s_1$ admissible and $s_0=s_1+1>\frac12$.  
 \begin{proposition}\label{prop:critical} For any admissible $s_0, s_1$ satisfying $s_0=s_1+1\in (\frac12,\frac52)$,   we have
\begin{align} 
\label{crit1} \|F\|_{X^{s_0,-\frac12}} & \les T^{\frac12-} \|u\|_{X^{s_0,\frac12} }  \|n\|_{Y^{s_1,\frac12+} },  \\
\label{crit2}
\Big\| \int_{\R } \la \lambda+\xi^2 \ra^{\frac{2s_0-3}{4}}   | \widehat{F}(\xi,\lambda)| d\xi  \Big\|_{L^2_\lambda} & \les T^{0+}\|u\|_{X^{s_0,\frac12}}  \|n\|_{Y^{s_1,\frac12+}}, \\  
\label{crit3} \|nu\|_{Z^{s_0}} & \les \|u\|_{X^{s_0,\frac12}}  \|n\|_{Y^{s_1,\frac12+}},
\end{align}
where $F=\eta(t/T)nu$ and 
$$
\|f\|_{Z^{s_0}}=\left\|\widehat f (\xi,\tau) \la \xi\ra^{s_0}\la \tau+\xi^2\ra^{-1}\right\|_{L^2_\xi L^1_\tau}. 
$$
\end{proposition}
The first two are enough to run the local theory without the continuity of the Schr\"odinger part of \eqref{eq:duhamel}; the third inequality gives the continuity.

\begin{proof}
We start with the proof of  \eqref{crit2}. Using \eqref{eq:xs3} and remark following it, it suffices to prove that 
$$
\Big\| \int_{\R } \la \lambda+\xi^2 \ra^{\frac{2s_0-3}{4}}   | \widehat{F}(\xi,\lambda)| d\xi  \Big\|_{L^2_\lambda}  \les  \| \eta(t/T)u\|_{X^{s_0,\frac12-}}  \|n\|_{Y^{s_1,\frac12+}}.
$$
Denoting  $\eta(t/T) u$ by $u$, as in the proof of Proposition~\ref{prop:smooth12},   it suffices to prove that 
$$
\sup_\tau \int_{\R^2 } \frac{\la \tau+\xi^2 \ra^{ s_0 -\frac32}  }{ \la \xi_1\ra^{2s_1} \la \xi-\xi_1\ra^{2s_0}   \la \tau- \xi_1 +(\xi-\xi_1)^2\ra^{1-} }   d\xi_1   d\xi<\infty.
$$
Consider the case $\frac12<s_0\leq \frac32$. It is easy to see the contribution of the resonant cases $|\xi_1|\les 1$ or $|\xi_1-2\xi|\les 1$  is bounded in $\tau$.
In the complement of this set we have
$$
\max(\la \tau+\xi^2 \ra, \la \tau- \xi_1 +(\xi-\xi_1)^2\ra)\gtrsim \la \xi_1\ra \la \xi_1-2\xi\ra.
$$
Therefore, we can bound the integral by 
\begin{multline*}
\int_{\R^2 } \frac{1  }{ \la \xi_1\ra^{2s_1+\frac32-s_0} \la \xi-\xi_1\ra^{2s_0}  \la \xi_1-2\xi\ra^{\frac32-s_0} \la \tau- \xi_1 +(\xi-\xi_1)^2\ra^{1-}  }   d\xi_1   d\xi \\ + \int_{\R^2 } \frac{1  }{ \la \xi_1\ra^{2s_1+1-} \la \xi-\xi_1\ra^{2s_0}  \la \xi_1-2\xi\ra^{1-} \la \tau+\xi^2 \ra^{\frac32- s_0}  }   d\xi_1   d\xi\\
\leq \int_{\R^2 } \frac{1  }{ \la \xi_1\ra^{2s_1+\frac32-s_0} \la \xi-\xi_1\ra^{2s_0}   \la \tau- \xi_1 +(\xi-\xi_1)^2\ra^{1-}  }   d\xi_1   d\xi 
\\ + \int_{\R^2 } \frac{1  }{ \la \xi_1\ra^{2s_1+1-} \la \xi-\xi_1\ra^{2s_0}  \la \xi_1-2\xi\ra^{1-}   }   d\xi_1   d\xi.
\end{multline*}
The second integral is finite by applying Lemma~\ref{lem:sums} to $\xi$ integral. We apply the change of variable $\eta=(\xi-\xi_1)^2$ in the first integral to obtain
$$
 \int_{\R^2 } \frac{1  }{ \la \xi_1\ra^{2s_1+\frac32-s_0} \la \eta \ra^{s_0} \sqrt{\eta}  \,  \la \tau- \xi_1 +\eta \ra^{1-}  }  d\eta  d\xi_1.
$$
Applying Lemma~\ref{lem:sums} in the $\eta$ integral, we bound this integral by
$$
 \int_{\R } \frac{1  }{ \la \xi_1\ra^{2s_1+\frac32-s_0}    \,  \la \tau- \xi_1  \ra^{1-}  }   d\xi_1 \les 1,
$$
since $2s_1+\frac32-s_0=s_1+\frac12>0$.

In the case $\frac32<s_0< \frac52$, using the inequality 
$$
\la \tau+\xi^2\ra \les \la \tau- \xi_1 +(\xi-\xi_1)^2\ra   \la \xi_1\ra \la \xi_1-2\xi\ra \les \la \tau- \xi_1 +(\xi-\xi_1)^2\ra   \la \xi_1\ra^2 \la \xi_1-\xi\ra,
$$
we bound the integral by
$$
  \int_{\R^2 } \frac{1 }{ \la \xi_1\ra  \la \xi-\xi_1\ra^{ s_0+\frac32}   \la \tau- \xi_1 +(\xi-\xi_1)^2\ra^{\frac52-s_0-}  }   d\xi_1   d\xi, 
$$
which is finite by the argument above using the change of variable $\eta=(\xi-\xi_1)^2$.

We now prove \eqref{crit1}. We first prove the following inequality 
$$
\Big|\int unh \ dx dt\Big|\les \|n\|_{Y^{s_1,\frac12+}} \left(\|u\|_{X^{s_0,\frac12}} \|h\|_{X^{-s_0,0}} + \|u\|_{X^{s_0,0}} \|h\|_{X^{-s_0,\frac12}}\right).
$$
Applying this to $n, \eta(t/T)u, $ and $\eta(t/2T)h$, then using the remark following \eqref{eq:xs3} we get  \eqref{crit1} by duality.

Noting that $\max(\langle \tau +\xi ^2 \rangle, \langle \tau - \tau_1   + (\xi - \xi_1  )^2 \rangle)\gtrsim \la \tau_1+\xi^2-(\xi - \xi_1  )^2\ra,$ it suffices to prove that
$$
\sup_{\xi,\tau} \int_{\R^2} \frac{ \langle\xi\rangle^{2s_0} \langle \xi_1 \rangle^{-2s_1}  \langle \xi -\xi_1 \rangle ^{-2s_0}}{  \langle \tau_1 -  \xi_1  \rangle ^{1+} \la \tau_1+\xi^2-(\xi - \xi_1  )^2\ra }  d\xi_1 d\tau_1 <\infty.
$$
Using Lemma~\ref{lem:sums} in the $\tau_1$ integral, we bound the supremum above by
\be\label{tempcrit}
\sup_{\xi } \int \frac{ \langle\xi\rangle^{2s_0} \langle \xi_1 \rangle^{-2s_1}  \langle \xi -\xi_1 \rangle ^{-2s_0}}{  \langle \xi^2+  \xi_1   - (\xi - \xi_1  )^2  \rangle }  d\xi_1=\sup_\xi \int \frac{ \langle\xi\rangle^{2s_0} \langle \xi_1 \rangle^{-2s_1}  \langle \xi -\xi_1 \rangle ^{-2s_0}}{  \langle \xi_1(\xi_1-2\xi)  \rangle }  d\xi_1.
\ee
In the equality we used the change of variable $\xi\to\xi- \frac12 $ as before. Note that we can restrict ourselves to the case $|\xi|\gg 1$, and that 
  the integral on $|\xi_1|<1$ is bounded in $\xi$. 
  
  Consider case $|\xi_1-2\xi|<1$. Letting $\rho=\xi_1(\xi_1-2\xi)$, we bound the integral on the set $|\xi_1-2\xi|<1$ by
$$
 \int \frac{ \langle\xi\rangle^{2s_0-2s_1-1} }{ \langle \rho \rangle   \langle \rho+\xi^2\rangle ^{ s_0}  }  d\rho =  \int \frac{ \langle\xi\rangle  }{ \langle \rho \rangle   \langle \rho+\xi^2\rangle ^{ s_0}  }  d\rho\les 1. 
$$
In the last inequality we used Lemma~\ref{lem:sums} and $s_0>\frac12$.

In the case $|\xi_1-2\xi|>1$, we bound the integral by 
$$
\int \frac{ \langle\xi\rangle^{2s_0} }{ \langle \xi_1 \rangle^{1+2s_1}  \langle \xi -\xi_1 \rangle ^{ 2s_0} \langle  \xi_1-2\xi \rangle }  d\xi_1 \les 1,
$$
by Lemma~\ref{lem:sums2}.

Finally, we prove \eqref{crit3}. By Cauchy-Schwarz inequality as before, it suffices to prove that
$$
\sup_\xi \int_{\R^3} \frac{ \langle\xi\rangle^{2s_0} \langle \xi_1 \rangle^{-2s_1}  \langle \xi -\xi_1 \rangle ^{-2s_0}}{\langle \tau +\xi^2 \rangle^2  \langle \tau_1 -  \xi_1  \rangle ^{1+}  \langle \tau - \tau_1   + (\xi - \xi_1  )^2  \rangle }  d\xi_1 d\tau_1 d\tau.
$$
By integrating in $\tau$ and $\tau_1$ using Lemma~\ref{lem:sums}, we bound this by
$$
\sup_\xi \int_{\R } \frac{ \langle\xi\rangle^{2s_0} \langle \xi_1 \rangle^{-2s_1}  \langle \xi -\xi_1 \rangle ^{-2s_0}}{  \langle -\xi^2 -   \xi_1   + (\xi - \xi_1  )^2  \rangle }  d\xi_1,
$$
which was handled before; see \eqref{tempcrit}. 
\end{proof}

\section{Local theory: The proof of  Theorem~\ref{thm:local}} \label{sec:prt3}

We first prove that the map $\Gamma=(\Gamma_1,\Gamma_2)$,
\begin{equation} \label{gamma}\begin{array}{l}
\Gamma_1(u,n)(t)=\eta(t) W_0^t\big(g_e, h  \big) -i \eta(t) \int_0^t e^{i(t- t^\prime)\partial_{xx}}   F(u,n) \,d t^\prime +i\eta(t) W_0^t\big(0,  q  \big), \\
\Gamma_2(u,n)(t)= \eta(t) V_0^t\big(\psi^{\pm}, f\big) + \frac12\eta(t) (n_+ +  n_-) - \frac12\eta(t) V_0^t(0,z),  
\end{array}
\end{equation}
has a fixed point in $X^{s_0,b}(\R)\times Y^{s_1,b}(\R)$. Here  $s_0\in(0,\frac52)$, $s_1\in(0,\frac32)$   satisfy  $$ 0<s_0-s_1< 1, \,\,\,\,2s_0>s_1+\frac12>0.$$    We choose $b<\frac12$ sufficiently close to $\frac12$. The critical case $s_0=s_1+1$ will be discussed later. Finally,
$F,q,n_\pm,z$ are given in equations \eqref{def:F}--\eqref{def:z}.

To see that $\Gamma$ is bounded in $X^{s_0,b}\times Y^{s_1,b}$ recall the following bounds:
Combining \eqref{eq:xs2}, \eqref{eq:xs3}, and Proposition~\ref{prop:smooth}, we obtain
\begin{multline*}
  \left\| \eta(t) \int_0^t e^{i(t- t^\prime)\partial_{xx}}   F(u,n) \,d t^\prime\right\|_{X^{s_0,b}} \les \|F(u,n)\|_{X^{s_0,-\frac12+}} \\
  \les T^{\frac12-b-} \|un\|_{X^{s_0,-b}}\les T^{\frac12-b-} \|u\|_{X^{s_0,b}} \|n\|_{Y^{s_1,b}}.
\end{multline*}
To obtain the analogous bound for the linear wave part first note that  
$$
\eta(t) W_0^t\big(g_e, h  \big)+i\eta(t) W_0^t\big(0,  q  \big)= \eta(t)e^{it\partial_{xx}} g_e+\eta(t) W_0^t\big( 0, h-p+iq \big). 
$$
By \eqref{eq:xs1}, we have
$$
\|\eta(t) e^{it\partial_{xx}} g_e\|_{X^{s_0,b}} \les \|g_e\|_{H^{s_0}}\les \|g\|_{H^{s_0} (\R^+)}.
$$
Using Proposition~\ref{prop:wbh} and Lemma~\ref{lem:Hs0} (noting that the compatibility condition holds) we have
\begin{multline} \label{eq:temp1}
 \| \eta(t) W_0^t\big( 0, h-p+iq \big)(t)\|_{X^{s_0,b}}
\les \|(h-p+iq)\chi_{(0,\infty)} \|_{H^{\frac{2s_0+1}{4}}_t(\R)} \\ \les \|h-p\|_{H^{\frac{2s_0+1}{4}}_t(\R^+)} + \|q\|_{H^{\frac{2s_0+1}{4}}_t(\R^+)}
\les \|h \|_{H^{\frac{2s_0+1}{4}}_t(\R^+)} +\|p\|_{H^{\frac{2s_0+1}{4}}_t(\R )}+\|q\|_{H^{\frac{2s_0+1}{4}}_t(\R )}.
\end{multline}
By Kato smoothing Lemma~\ref{lem:kato}, we have
$$
\|p\|_{H^{\frac{2s_0+1}{4}}_t(\R )} \les \|g\|_{H^{s_0}(\R^+)}.
$$
By combining  Proposition~\ref{prop:duhamelkato}, \eqref{eq:xs3},  Proposition~\ref{prop:smooth} and  Proposition~\ref{prop:smooth12} we have
$$
\|q\|_{H^{\frac{2s_0+1}{4}}_t(\R )} 
\les T^{ \frac12-b-} \| u\|_{X^{s_0,b}} \|n\|_{Y^{s_1,b}}.
$$
Combining these estimates, we obtain
$$
\|\Gamma_1 (u,n)\|_{X^{s_0,b}}\les \|g\|_{H^{s_0}(\R^+)}+ \|h \|_{H^{\frac{2s_0+1}{4}}_t(\R^+)} + T^{ \frac12-b-} \| u\|_{X^{s_0,b}} \|n\|_{Y^{s_1,b}}.
$$
Similarly, using the analogs of  \eqref{eq:xs1},  \eqref{eq:xs2}, \eqref{eq:xs3} for $Y^{s_1,b}_\pm$ norms, and Proposition~\ref{prop:smooth2} we obtain
$$
\|\eta(t) (n_+ +  n_-)\|_{Y^{s_1,b}}\leq \|\eta(t)  n_+  \|_{Y^{s_1,b}_+}+\|\eta(t)  n_-  \|_{Y^{s_1,b}_-}  
\les T^{\frac12-b-} \|u\|_{X^{s_0,b}}^2.
$$ 
For the linear part, first note that
$$
\eta(t)V_0^t(\psi^\pm, f)-\frac12 \eta(t)V_0^t(0, z)=   \frac12 \eta(t) \left[e^{  t \partial_x} \psi_{+}+ e^{- t \partial_x} \psi_{-}\right] + \eta(t) V_0^t(0, f-r- z/2).
$$
Using   Proposition~\ref{prop:wbhW},  and Lemma~\ref{lem:Hs0} (noting that the compatibility condition holds) we have
\begin{multline} \label{eq:temp12}
 \| \eta(t) V_0^t\big( 0,  f-r- z/2 \big)(t)\|_{Y^{s_1,b}}
\les \|(f-r-z/2)\chi_{(0,\infty)} \|_{H^{s_1}_t(\R)} \\ \les \|f-r\|_{H^{s_1}_t(\R^+)} + \|z\|_{H^{s_1}_t(\R^+)}
\les \|f \|_{H^{s_1}_t(\R^+)} +\|r\|_{H^{s_1}_t(\R )}+\|z\|_{H^{s_1}_t(\R )}.
\end{multline}
By Kato smoothing Lemma~\ref{lem:katoW}, we have
$$
\|r\|_{H^{s_1}_t(\R )} \les \|\psi_+\|_{H^{s_1}_t(\R )}+\|\psi_-\|_{H^{s_1}_t(\R )}\les  \|n_0\|_{H^{s_1}(\R^+)} +\|n_1\|_{\hat H^{s_1-1}(\R^+)}.
$$
Finally, by combining  Proposition~\ref{prop:duhamelkatoW}, \eqref{eq:xs3},  Proposition~\ref{prop:smooth2}, Proposition~\ref{prop:smooth22}, and Proposition~\ref{prop:smooth23} we have
$$
\|z\|_{H^{s_1}_t(\R )} 
\les T^{ \frac12-b-} \| u\|_{X^{s_0,b}}^2.
$$
Combining these estimates, and \eqref{eq:xs1},  we obtain
$$
\|\Gamma_2 (u,n)\|_{Y^{s_1,b}}\les \|f \|_{H^{s_1}_t(\R^+)}+ \|n_0\|_{H^{s_1}(\R^+)} +\|n_1\|_{\hat H^{s_1-1}(\R^+)} + T^{ \frac12-b-} \| u\|_{X^{s_0,b}}^2.
$$

The estimates on the differences are similar. We thus obtain  the existence of a fixed point $(u,n)$ in $X^{s_0,b}\times Y^{s_1,b}$.
Now we prove that $u\in C^0_t H^{s_0}_x([0,T)\times \R)$. Note that continuity of the  first and the third term in the definition of $\Gamma_1$  is continuous in $H^{s_0}$ by Lemma~\ref{lem:wbcont}  and \eqref{eq:temp1}. For the second term it follows from the embedding  $X^{s_0,\frac12+}\subset C^0_tH^{s_0}_x$
 and \eqref{eq:xs2} together with Proposition~\ref{prop:smooth}. The fact that  $u\in C^0_x H^{\frac{2s_0+1}{4}}_t(\R\times [0,T])$ follows similarly from Lemma~\ref{lem:kato}, Proposition~\ref{prop:duhamelkato}, and Lemma~\ref{lem:wbcont}. The analogous statements for $n$ are proved similarly. 
The continuous dependence on the initial and boundary data follows from the fixed point argument and the linear and nonlinear estimates as above.

In the critical case $s_0=s_1+1$, for the fixed point argument, one needs to consider the space $X^{s_0,\frac12}\times Y^{s_1,\frac12+}$. The required bounds are provided in Section~\ref{sec:critical}. The continuity of the Schr\"odinger part of the solution uses in addition the bound we have for $Z^{s_1}$ norm in Proposition~\ref{prop:critical}. For more details see \cite{Bou2,GTV}.


\subsection{Uniqueness of smooth solutions}\label{sec:unique} In this section we discuss the uniqueness of solutions of \eqref{eq:zakharov}.
The solution we constructed above is the unique fixed point of \eqref{gamma}. However, it is not a priori clear if different extensions of initial data produce  the same solution on $\R^+$.  We start with a proof of uniqueness in the case $s_0=2$, $s_1=1$. It is clear that the uniqueness of smoother solutions follows from this. The uniqueness of rougher solutions (under additional assumptions) will follow from an approximation argument.

Let $(u_1,n_1,\nu_1)$, $ (u_2,n_2,\nu_2)$, where $\nu_j=\partial_x^{-1} n_t$, $j=1,2$,  be two local solutions constructed as above starting from the same initial data but possibly with different extensions to $\R$. The following energy calculations are variants of the ones obtained in \cite{ss,aa2,GZZ} 

Let 
$$
u=u_1-u_2,\,\,\,n=n_1-n_2,\,\,\,\nu=\nu_1-\nu_2.
$$
Note that $u,n,\nu$ satisfies the following system:
\begin{equation}\label{eq:zak_unique}
\left\{
\begin{array}{l}
iu_{t}+  u_{xx} =n_1u_1-n_2u_2, \,\,\,\,  x \in { \R}^+=(0,\infty), \,\,\,\,  t\in \R^+,\\
n_t=\nu_x,\\
\nu_{t}=n_{x}+(|u_1|^2-|u_2|^2)_{x}, \\
u(x,0)=n(x,0)=\nu(x,0)=\nu_x(x,0)=0, \\
u(0,t)=n(0,t)=\nu_x(0,t)=0.
\end{array}
\right.
\end{equation}
Let 
$$
L(t):= \|u_x\|^2_{L^2(\R^+)} +\frac12\| u\|_{L^2(\R^+)}^2+\frac12 \|n\|_{L^2(\R^+)}^2+\frac12\|\nu\|_{L^2(\R^+)}^2+ \int_0^\infty n(|u_1|^2-|u_2|^2).
$$ 
We have 
\begin{multline*}
\partial_t\left(\frac12\| u\|_{L^2}^2+\frac12 \|n\|_{L^2}^2+\frac12\|\nu\|_{L^2}^2\right) = \Re \int_0^\infty\overline u u_t +\int_0^\infty nn_t+\int_0^\infty \nu \nu_t\\
=\Im  \int_0^\infty\overline u (n_1u_1-n_2u_2)+\int_0^\infty nn_t+\int_0^\infty \nu \big[(n_x+(|u_1|^2-|u_2|^2)_x\big]\\
=\Im  \int_0^\infty\overline u n  u_2 -\int_0^\infty \nu_x  (|u_1|^2-|u_2|^2). 
\end{multline*}
The last equality follows from integration by parts. Similarly, we have
$$
\partial_t \| u_x\|_{L^2}^2  = -2\Im   \int_0^\infty\overline u_{xx} (n_1u_1-n_2u_2).
$$
Also note that
\begin{multline*}
\partial_t\int_0^\infty n(|u_1|^2-|u_2|^2) =\int_0^\infty n_t(|u_1|^2-|u_2|^2) +2\Re \int_0^\infty n(\overline{u_1}u_{1t}-\overline{u_2}u_{2t}) \\
=\int_0^\infty \nu_x(|u_1|^2-|u_2|^2) -2 \Im \int_0^\infty n(\overline{u_1}u_{1xx}-\overline{u_2}u_{2xx})\\
=\int_0^\infty \nu_x(|u_1|^2-|u_2|^2) +2 \Im \int_0^\infty n u_1 \overline{u}_{xx} +2 \Im \int_0^\infty n u  \overline{u_2}_{xx}.
\end{multline*}
In the last identity we added and subtracted $n\overline{u_1}u_{2xx}$ from the integrand and took complex conjugates.
Combining these identities, and then integrating by parts, we obtain
\begin{multline*}
\frac{dL}{dt}= \Im  \int_0^\infty\overline u  n u_2 +2\Im   \int_0^\infty\overline u_{xx}  n_2u+2 \Im \int_0^\infty n u  \overline{u_2}_{xx}\\
= \Im  \int_0^\infty\overline u  n u_2 -2\Im   \int_0^\infty\overline u_x  n_{2x}u+2 \Im \int_0^\infty n u  \overline{u_2}_{xx} 
\les \|u\|_{H^1}^2+\|n\|_{L^2}^2.
\end{multline*}
Here the implicit constant depend on $H^2\times H^1$ norm of the local solutions.

We now estimate the mixed term in $L$ directly. 
First note that
$$
\partial_t\big(|u_1|^2-|u_2|^2\big)=-2\Im \left(  \overline u_1 u_{1xx} -\overline u_2 u_{2xx} \right)=-2\Im \partial_x\left(\overline u_1 u_{1x} -\overline u_2 u_{2x} \right).
$$ 
Therefore 
\begin{multline*}
\partial_t \big\||u_1|^2-|u_2|^2\big\|_2^2= -2\Im \int_0^\infty \big(|u_1|^2-|u_2|^2\big)\partial_x\left(\overline u_1 u_{1x} -\overline u_2 u_{2x} \right)\\
= 2\Im \int_0^\infty  \left(\overline u_1 u_{1x} -\overline u_2 u_{2x} \right)^2 = 2\Im \int_0^\infty  \left(\overline u  u_{1x} +\overline u_2 u_{ x} \right)^2. 
\end{multline*}
By Cauchy-Schwarz and integrating in time, this implies that 
$$
\big\||u_1|^2-|u_2|^2\big\|_2\les \sqrt{\int_0^t \|u(t^\prime)\|_{H^1}^2 dt^\prime}.
$$
Therefore
$$
\left|\int_0^\infty n(|u_1|^2-|u_2|^2)\right|\leq\|n\|_2 \big\||u_1|^2-|u_2|^2\big\|_2 \les \|n\|_2 \sqrt{\int_0^t \|u(t^\prime)\|_{H^1}^2 dt^\prime}.
$$
Let 
$$\alpha(t)=\| u\|_{H^1}^2+ \|n\|_{L^2}^2+ \|\nu\|_{L^2}^2.$$
The calculations above (noting that $L(0)=\alpha(0)=0$) imply that
$$
\alpha(t)\les \int_0^t\alpha(t^\prime) dt^\prime + \sqrt{\alpha(t)} \sqrt{ \int_0^t \alpha(t^\prime)} dt^\prime. 
$$
Therefore $\alpha(t)=0$ in the local existence interval, which implies uniqueness in $H^2\times H^1\times \hat H^0$.

\section{Global wellposedness}
In this section we  obtain energy identities that imply a priori bounds for   $H^1\times L^2\times \hat H^{-1} (\R^+)$ norm of the solutions of \eqref{eq:zakharov} when $h=0$. Global wellposedness of smoother solutions (when $h=0$) and polynomial bounds on the Sobolev norms follows from this by
the smoothing estimates we obtained in the previous section, see \cite{et3, et4}. 

We define $\nu\in L^2(\R^+)$ satisfying
$$
n_t=\nu_x,\,\,\,\,\,\,\,\nu_t=(n+|u|^2)_x.
$$
Let 
$$
E(t):= \|u_x\|^2_{L^2}+ \int_0^\infty n|u|^2 dx +\frac12 \|n\|_{L^2}^2+\frac12\|\nu\|_{L^2}^2.
$$

Note that by Gagliardo--Nirenberg inequality 
\be\label{enerbound}
 \|u_x\|^2_{L^2}+ \|n\|_{L^2}^2+ \|\nu\|_{L^2}^2 \les E(t)+\|u_x\|_{L^2} \|u\|_{L^2}^3,
\ee
and 
\be\label{enerbound2}
 E(t)  \les \|u_x\|^2_{L^2}+ \|n\|_{L^2}^2+ \|\nu\|_{L^2}^2+\|u_x\|_{L^2} \|u\|_{L^2}^3.
\ee

\begin{proposition}\label{prop:energy}
We have the following a priori estimates for the solutions of \eqref{eq:zakharov} when $h=0$:
$$
\|u\|_{H^1}+\|n\|_{L^2}+\|n_t\|_{\hat H^{-1}} \leq C\left(\|g\|_{H^1},\|n_0\|_{L^2},\|n_1\|_{\hat H^{-1}}, \|f\|_{L^2}\right).
$$
\end{proposition}
\begin{proof}
The following identities can be justified by approximation by $H^2\times H^1\times L^2 $ solutions:
\begin{align}
\label{eq:m} &\partial_t|u|^2= -2 \Im(u_x\overline{u})_x ,\\
\label{eq:e}&\partial_t\left(|u_x|^2+n|u|^2+\frac12n^2+\frac12\nu^2\right) =[\nu(n+|u|^2)]_x+2\Re(u_x\overline{u_t})_x,\\
\label{eq:l}&\partial_x\left(|u_x|^2+n^2+\nu^2\right)=2[\nu n]_t-2i\big[(u\overline{u_x})_t-(u\overline{u_t})_x\big].
\end{align}
We start by estimating $\|u\|_{L^2}$. By integrating \eqref{eq:m}   in $ [0,\infty)\times [0,t]$, we obtain
$$
\int_0^\infty |u(x,t)|^2 dx = \int_0^\infty |g(x )|^2 dx+2\Im \int_0^t u_x(0,s) \overline{h(s)} ds=  \int_0^\infty |g(x )|^2 dx.
$$
Using this in \eqref{enerbound} and \eqref{enerbound2}, we conclude that 
$$
E(t)\les  \|u_x\|^2_{L^2}+ \|n\|_{L^2}^2+ \|\nu\|_{L^2}^2+1,
$$
\be\label{enerbound4}
  \|u_x\|^2_{L^2}+ \|n\|_{L^2}^2+ \|\nu\|_{L^2}^2\les E(t)+1,
\ee
where implicit constants depend on $\|g\|_{L^2}$. 
To estimate $E(u,n,\nu)$, we integrate \eqref{eq:e} in $ [0,\infty)\times [0,t]$:
\begin{multline*}
E(t)-E(0) = -\int_0^t \nu(0,s) [f(s) + |h(s)|^2] ds   
-2\Re \int_0^t u_x(0,s) \overline{h^\prime(s)} ds \\
= - \int_0^t \nu(0,s)  f(s)  ds    .
\end{multline*}
Using  Cauchy-Schwarz inequality  we have
\be\label{eq:Ebound}
E(t)-E(0) \les \|f\|_{L^2_{[0,t]}}  J^{1/2}\les J^{1/2},
\ee
where 
$$
J:=\int_0^t|\nu(0,s)|^2 ds.
$$  
 
Finally, to estimate   $J$, we  integrate \eqref{eq:l} in  $ [0,\infty)\times [0,t]$, and obtain the following inequality by discarding the positive terms on the left hand side:
\begin{multline*}
J \leq -2 \int_0^\infty \left[\nu(x,t) n(x,t) -\nu(x,0)n(x,0)\right]dx  \\ + 2i \int_0^\infty u(x,t)\overline{u_x(x,t)} dx - 2 i \int_0^\infty g(x)\overline{g^\prime(x)} dx +2i\int_0^t
h(s)\overline{h^\prime(s)} ds.
\end{multline*}
And hence
\begin{multline*}   J \leq  2\|\nu \|_{L^2_x} \|n \|_{L^2_x}+ 2\|\nu(x,0)\|_{L^2_x} \|n_0\|_{L^2_x}  +2\|u\|_{L^2_x}\|u_x\|_{L^2_x}  + 2\|g\|_{L^2}\|g\|_{H^1}\\
=2\|\nu \|_{L^2_x} \|n \|_{L^2_x}+ 2\|\nu(x,0)\|_{L^2_x} \|n_0\|_{L^2_x}  +2\|g\|_{L^2_x}\|u_x\|_{L^2_x}  + 2\|g\|_{L^2}\|g\|_{H^1}\\
\les E(t)+ E(0)+1.  
\end{multline*}
Using this in \eqref{eq:Ebound}, we have 
$$
E(t)-E(0) \les  \sqrt{E(t)+ E(0)+1},
$$ 
where the implicit constant depend only on boundary data. Therefore
$$
E(t) \les  E(0) +1. 
$$
Using \eqref{enerbound4} we conclude that 
$\|u\|_{H^1}+ \|n\|_{L^2}+ \|\nu\|_{L^2} $ remains bounded. 
\end{proof}

\section{Uniqueness of rough solutions}\label{sec:uniquerough}

We now prove uniqueness in rougher norms in the case $h=0$. It suffices to work with admissible $(s_0,s_1)$ satisfying $0<s_0<\frac12$, $-\frac12<s_1<0$.
Let $\tilde g$ , $\tilde{\tilde g}$ be two $H^{s_0}(\R)$ extensions of $g\in H^{s_0}(\R^+)$, $s_0\in (0,1/2)$. 
Similarly take extensions $\tilde n_0$ , $\tilde{\tilde n}_0$ of $n_0$ in $H^{s_1}$. Note that $n_1=\partial_x \nu$ for some $\nu\in H^{s_1}(\R^+)$. 
We also take extensions of $n_1$ as  $\tilde n_1=\partial_x \tilde \nu$, $\tilde{\tilde n}_1=\partial_x\tilde{\tilde \nu}$ where   $\tilde \nu, \tilde{\tilde \nu}$ are $H^{s_1}$ extensions of $\nu$.

Take a sequence $g_k \in H^2(\R^+)$ converging to $g$ in $H^{s_0}(\R^+)$.
Let $\tilde g_k, \tilde{\tilde g}_k \in H^2(\R)$ be extensions of $g_k$ converging to $\tilde g$ , $\tilde{\tilde g}$ 
in $H^r(\R)$ for $r<s_0$, see Lemma~\ref{lem:ext} below. Similarly construct sequences converging to $\tilde n_0$ , $\tilde{\tilde n}_0$ in $H^{r}(\R)$, $r<s_1$, and to $\tilde n_1$ , $\tilde{\tilde n}_1$ in $\hat H^{r}(\R)$, $r<s_1-1$.

Also take a sequence   $f_k\in H^{1}(\R^+)$ converging to $f$ in $H^{s_1}(\R^+)$. Construct global solutions $(\tilde u_k, \tilde n_k) $  and $(\tilde {\tilde u}_k, \tilde {\tilde n}_k) $ in $H^2(\R)\times H^1(\R)$ with $h=0$. By the uniqueness statement in Section 5, the restriction of the  solutions $(\tilde u_k, \tilde n_k) $  and $(\tilde {\tilde u}_k, \tilde {\tilde n}_k) $   to $\R^+$ are the same. Since, by the fixed point argument,  the solutions $(\tilde u , \tilde n ) $  and $(\tilde {\tilde u} , \tilde {\tilde n} ) $
 are the limits of these solutions in $H^{s_0-}\times H^{s_1-}$, their restriction to $\R^+$ are the same.

\begin{lemma}\label{lem:ext} Fix $-\frac12<s<\frac12$  and $k\in \Z^+$. 
Let $g\in H^s(\R^+)$, $f\in H^k(\R^+)$, and let $g_e$ be an $H^s$ extension of $g$ to $\R$. Then there is an $H^k$ extension $f_e$ of $f$ to $\R$ so that 
$$\|g_e-f_e\|_{H^r(\R)}\les \|g-f\|_{H^s(\R^+)},\,\,\text{ for } r<s.
$$ 
\end{lemma}
\begin{proof}
Fix $-\frac12<s<\frac12$ and $k\in \Z^+$. We start with the following \\
Claim. Fix $\psi \in H^s(\R)$ supported in $(-\infty,0]$. For any $\epsilon>0$, there is a function $\phi\in H^k(\R)$ supported in $(-\infty,0)$ such that $\|\phi-\psi\|_{H^r(\R)}<\epsilon$ for $r<s$.\\
To prove this claim first consider the case $0\leq s <\frac12$. Note that 
 $\chi_{(-\infty,-\delta)} \psi \to \psi $ 
in $L^2(\R)$ as $\delta \to 0^+$. Also note by Lemma~\ref{lem:Hs0} that $\|\chi_{(-\infty,-\delta)} \psi\|_{H^s(\R)}\les \|\psi\|_{H^s(\R)}$ uniformly in $\delta$. Therefore
 $\chi_{(-\infty,-\delta)} \psi \to \psi $ in $H^r$ for $r<s$ by interpolation. Now note that by duality $\chi_{(-\infty,-\delta)} \psi \to \psi $ in $H^r$ for $r<s$ also in the case $-\frac12 < s<0$.
 
The claim follows by taking a smooth approximate identity $\rho_n$ supported in $(-\delta,\delta)$ for sufficiently small $\delta$, and letting
$\phi= [\chi_{(-\infty,-\delta)} \psi ] * \rho_n$ for sufficiently large $n$. 

To obtain the lemma from this claim, let $\widetilde f$ be an $H^k$ extension of $f$ to $\R$, and let $h$ be an $H^s$ extension of $g-f$ to $\R$ with $\|h\|_{H^s(\R)}\les \|g-f\|_{H^s(\R^+)}$. Apply the claim to $\psi=g_e-\widetilde f-h$ with $\epsilon=\|g-f\|_{H^s(\R^+)}$. Letting $f_e=\widetilde f +\phi$ yields the claim.  
\end{proof}

\section{Appendix}\label{sec:appendix}

In this appendix we provide the following lemmas   that we used throughout the proof. The first one can be found in \cite{et3}. 
\begin{lemma}\label{lem:sums} If $\beta \geq \gamma \geq 0$ and $\beta + \gamma > 1$, then
\[ \int \frac{1}{\langle x-a_1 \rangle ^\beta \langle x - a_2 \rangle^\gamma} dx \lesssim \langle a_1-a_2 \rangle ^{-\gamma} \phi_\beta(a_1-a_2),  \]
where
\[ \phi_\beta(a)=\begin{cases} 1 & \beta > 1 \\  \log(1 + \langle a \rangle ) & \beta = 1 \\ \langle a \rangle ^{1-\beta} &\beta < 1 .\end{cases} \]
\end{lemma}

\begin{lemma}\label{lem:sums2} If $\alpha, \beta, \gamma \geq 0$ and $\alpha+\beta + \gamma > 1$. Let $\ell:=\max(1,\alpha,\beta,\gamma)$. Then
\[ \int \frac{1}{\langle x- a \rangle^\alpha \langle x \rangle^\beta \la x+a\ra^\gamma} dx \lesssim \langle a \rangle^{-\alpha-\beta-\gamma+\ell }  \la a \ra^{0+},  \]
the term $\la a \ra^{0+}$ can be discarded unless $\max(\alpha,\beta,\gamma)=1$.
\end{lemma}
 \begin{proof}
The inequality is clear in the case $|a|\les 1$. For $|a|\gg 1$,  the inequality follows by dividing the integral into three pieces: $|x|>2|a|$, $|x|<|a|/2$, 
$|a|/2<|x|<2|a|$.
 \end{proof}

\end{document}